\documentclass[11pt]{amsart}
\usepackage{amsmath}
\usepackage{amsfonts}
\usepackage{amssymb}
\newtheorem{thm}{Theorem}[section]

\newtheorem{cor}[thm]{Corollary}

\newtheorem{remark}[thm]{Remark}
\newtheorem{lemma}[thm]{Lemma}
\newtheorem{prop}[thm]{Proposition}

\newtheorem{ex}[thm]{Example}
\newtheorem{defn}[thm]{Definition}

\newcommand{\bb}[1]{\mathbb{#1}}
\newcommand{\cl}[1]{\mathcal{#1}}
\newcommand{\ff}[1]{\mathfrak{#1}}

\begin{document}

\title[Algebras of Functions]{Operator Algebras of Functions}

\author[M.~Mittal]{Meghna Mittal}
\address{Department of Mathematics, University of Houston, Houston,
  Texas 77204-3476, U.S.A.}
\email{mittal@math.uh.edu}

\author[V.~I.~Paulsen]{Vern I.~Paulsen}
\address{Department of Mathematics, University of Houston,
Houston, Texas 77204-3476, U.S.A.}
\email{vern@math.uh.edu}

\thanks{}

\begin{abstract}  We present some general theorems
  about operator algebras that are algebras of functions on sets,
  including theories of local algebras, residually finite dimensional
  operator algebras and algebras that can be represented as
  the scalar multipliers of a vector-valued reproducing kernel Hilbert
  space.   We use these to further develop a quantized function theory
  for various domains that extends and unifies Agler's theory
  of commuting contractions and the Arveson-Drury-Popescu theory of commuting
  row contractions. We obtain analogous factorization theorems, prove
  that the algebras that we obtain are dual operator algebras and show
  that for many domains, supremums over all commuting tuples of operators
  satisfying certain inequalities are obtained over all commuting
  tuples of matrices. 

\end{abstract}

\maketitle


  \section{Introduction}

  A concrete {\em operator algebra} $\cl A$ is just a subalgebra of $
  B(\cl H),$ the bounded operators on a Hilbert space $\cl H.$ The
  operator norm on $ B(\cl H)$ gives rise to a norm on $\cl A.$
  Moreover, the identification $M_n(\cl A) \cong \cl A \otimes M_n
  \subseteq  B(\cl H \otimes \bb C^n) \cong B(\cl H^n)$ endows
  the matrices over $\cl A$ with a family of norms in a natural way,
  where $M_n$ denotes the algebra of $n \times n$ matrices. It is common
  practice to identify two operator algebras $\cl A$ and $\cl B$ as being the ``same'' if
  and only if there exists an algebra isomorphism $\pi: \cl A \to \cl B$
  that is not only an isometry, but which also preserves all the matrix
  norms, that is such that $\|(\pi(a_{i,j}))\|_{M_n(\cl B)} =
  \|(a_{i,j})\|_{M_n(\cl A)},$ for every $n$ and every element
  $(a_{i,j}) \in M_n(\cl A).$  Such a map $\pi$ is called a {\em
  completely isometric isomorphism.} In \cite{BRS} an abstract
  characterization of operator algebras, in the above sense, was given
  and since that time a theory of such algebras has evolved. For more
  details on the abstract theory of operator algebras, see \cite{BLM}, \cite{Pa} or \cite{Pi}.

  In this note we present a theory for a special class of abstract
  abelian operator
  algebras which contains many important examples arising in function
  theoretic operator theory, including the Schur-Agler and the
  Arveson-Drury-Popescu algebras. This theory will allow us to answer certain
  types of questions about such algebras in a unified manner. We will
  prove that our hypotheses give an abstract characterization of
  operator algebras that are completely isometrically isomorphic to
  multiplier algebras of vector-valued reproducing kernel Hilbert spaces.

  Our results will show that under certain mild
  hypotheses, operator algebra norms which are defined by taking the
  supremum of certain families of operators on Hilbert spaces of
  arbitrary dimensions can be obtained by restricting the family of
  operators to finite dimensional Hilbert spaces.
  Thus, in a certain sense, which will be explained later, our results
  give conditions that guarantee that an algebra is {\em residually
  finite dimensional.}

Finally, we apply our results to study ``quantized function theories''
on  various domains.  Our work in this direction should be compared to
that of Ambrozie-Timotin\cite{AT}, Ball-Bolotnikov\cite{BB} and
Kalyuzhnyi-Verbovetzkii\cite{K-V}. These authors obtain the same type
of factorization theorem via unitary colligation methods as we obtain
via operator algebra methods, but they assume somewhat different(and
in many cases stronger) hypotheses. We, also, obtain a bit more information 
about the algebras themselves, including the fact that in many cases their 
norms can be obtained as supremums over matrices instead of operators. 

  We now give the relevant definitions. Recall that given any set $X$
  the set of all complex-valued functions on $X$ is an algebra over the
  field of complex numbers.

  \begin{defn} We call 
  $\cl A$ an {\bf operator algebra of functions} on a set $X$ provided:
  \begin{enumerate}
  \item $\cl A$ is a subalgebra of the algebra of functions on $X,$
  \item $\cl A $ separates the points of X and contains the constant functions,
 
  \item for each $n$ $M_n(\cl A)$ is equipped with a norm $\|.\|_{M_n(\cl
   A)},$ such that the set of norms satisfy the BRS axioms\cite{BRS} to be 
   an abstract operator algebra,
  \item for each $x \in X,$ the evaluation functional, $\pi_x: \cl A \to
  \bb C,$ given by $\pi_x(f) = f(x)$ is bounded.
  \end{enumerate}
  \end{defn}

  A few remarks and observations are in order.
  First note that if $\cl A$ is an operator algebra of functions on $X$
  and $\cl B \subseteq \cl A$ is any subalgebra, which contains the
  constant functions and still separates points, then $\cl B,$ equipped
  with the norms that $M_n(\cl B)$ inherits as a subspace of $M_n(\cl
  A)$ is still an operator algebra of functions.
 
  The basic example of an operator algebra of functions is
  $\ell^{\infty}(X),$ the algebra of all bounded functions on $X.$ If for
  $(f_{i,j}) \in M_n(\ell^{\infty}(X))$ we set
  \[ \|(f_{i,j})\|_{M_n(\ell^{\infty}(X))} = \|(f_{i,j})\|_{\infty} \equiv
  \sup \{ \|(f_{i,j}(x))\|_{M_n}: x \in X \}, \] where $\| \cdot\|_{M_n}$
  is the norm on $M_n$ obtained via the identification $M_n= B(\bb
  C^n),$ then it readily follows
  that properties (1)--(4) of the above definition are satisfied.
  Thus, $\ell^{\infty}(X)$ is an operator algebra of functions on $X$ in our
  sense and any subalgebra of $\ell^{\infty}(X)$ that contains the
  constants and separates points will be an operator algebra of
  functions on $X$ when equipped with the subspace norms.

  \begin{prop} Let $\cl A$ be an operator algebra of functions on $X,$
  then $\cl A \subseteq \ell^{\infty}(X),$ and for every $n$ and every
  $(f_{i,j}) \in M_n(\cl A),$ we have
  $\|(f_{ij})\|_{\infty} \leq \|(f_{ij})\|_{M_n(\cl A)}.$
  \end{prop}
  \begin{proof} Since $\pi_x: \cl A \to \bb C$ is bounded and the norm
  is sub-multiplicative, we have that for any $f \in \cl A,$
  $|f(x)|^n = |\pi_x(f^n)| \le \|\pi_x\|\|f^n\| \le \|\pi_x\|\|f\|^n.$
  Taking the $n$-th root of each side of this inequality and letting $n
  \to +\infty,$ yields $|f(x)| \le \|f\|,$ and hence, $f \in
  \ell^{\infty}(X).$ Note also that $\|\pi_x\|=1.$

  Finally, since every bounded, linear functional on an operator space
  is completely bounded and the norm and the cb-norm are equal, we have
  that $\|\pi_x\|_{cb} = \|\pi_x\| =1.$  Thus, for $(f_{i,j}) \in
  M_n(\cl A),$ we have $\|(f_{i,j}(x))\|_{M_n} = \|(\pi_x(f_{i,j}))\|
  \le \|(f_{i,j})\|_{M_n(\cl A)}.$
  \end{proof}

  Given an operator algebra of functions on a set $X,$ $\cl A$ and $F=
  \{ x_1,...,x_k\}$ a set of distinct points in X, we define $I_F=\{f \in \cl A: f(z)=0 \ \forall \ z
  \in F\}$ and note that $ M_n(I_F) =\{f \in M_n(\cl A): f(z)=0 \ \forall \ z
  \in F\} \ \forall \ n.$ The quotient space $\cl A/I_F$ has a natural
  set of matrix norms given by defining 
  $ \|(f_{i,j} + I_F)\| = \inf \{ \|(f_{i,j} +g_{i,j})\|_{M_n(\cl A)}
  : g_{i,j} \in I_F \}. $
  Alternatively, this is the norm on $M_n(\cl A/I_F)$ that comes
  via the identification, $M_n(\cl A/I_F) = M_n(\cl A)/M_n(I_F),$
  where the latter space is given its quotient norm. It is
  easily checked that this family of matrix norms satisfies the BRS
  conditions and so gives $\cl A/ I_F$ the structure of an abstract
  operator algebra.

  We let $\pi_F(f) = f+ I_F$ denote the  quotient map  $\pi_F: \cl A \rightarrow \cl
  A/I_F$ so that for each $n, \pi_F^{(n)}: M_n(\cl A) \rightarrow
  M_n(\cl A/I_F) \cong M_n(\cl A)/M_n(I_F).$

  Since $\cl A$ is an algebra which separates points on $X$
  and contains constant functions, it follows that there exists
  functions $f_1,...,f_k \in \cl A,$ such that
  $f_i(x_j) = \delta_{i,j},$ where $\delta_{i,j}$ denotes the Dirac
  delta function. If we set $E_j = \pi_F(f_j),$ then it is easily seen
  that whenever $f \in \cl A$ and $f(x_i) = \lambda_i, i =1,...,k,$ then
  $\pi_F(f) = \lambda_1E_1 + \cdots + \lambda_k E_k.$ Moreover, $E_iE_j
  = \delta_{i,j} E_i,$ and $E_1 + \cdots + E_k = 1,$ where $1$ denotes
  the identity of the algebra $\cl A/I_F.$ Thus, $\cl A/I_F = span \{
  E_1,...,E_k \},$ is a unital algebra spanned by $k$ commuting
  idempotents.
  Such algebras were called {\em k-idempotent operator algebras} in
  \cite{Pa2} and we will use a number of results from that paper.

  \begin{defn}
  An operator algebra of functions $\cl A$ on a set X, is called a 
  {\bf local operator algebra of functions} if it satisfies
  \begin{center}
  $sup_F\|\pi^{(n)}_F((f_{ij}))\|=\|(f_{ij})\| \ \forall \ (f_{ij}) 
  \in M_n(\cl A) $ and for every n,
  \end{center}
  where the supremum is taken over all finite subsets of X.
  \end{defn}

  The following result shows that every operator algebra of functions
  can be re-normed so that it becomes local.

  \begin{prop} Let $\cl A$ be an operator algebra of functions on X, let
  $\cl A_L = \cl A$ and define a family of matrix norms on $\cl A_L,$
  by setting $\|(f_{i,j})\|_{M_n(\cl A_L)} = \sup_F
  \|(\pi_F(f_{i,j}))\|_{M_n(\cl A/I_F)}.$  Then $\cl A_L$ is a local
  operator algebra of functions on X and the identity map, $id: \cl A
  \to \cl A_L,$ is completely contractive.
   \end{prop}
  \begin{proof} It is clear from the definition of the norms on $\cl
  A_L$ that the identity map is completely contractive and it is readily
  checked that $\cl A_L$ is an operator algebra of functions on X.

  Let $\tilde{\pi}_F: \cl A_L \to \cl A_L/I_F,$ denote the
  quotient map, so that $\|\tilde{\pi}_F(f)\| = \inf \{ \|f+g\|_L: g
  \in I_F \} \le \inf \{ \|f+g\| : g \in I_F \} = \|\pi_F(f)\|,$ since
  $\|f+g\|_L \le \|f+g\|.$ We claim that for any $f \in \cl A,$ and any
  finite subset $F \subseteq X,$ we have that $\|\pi_F(f)\| =
  \|\tilde{\pi}_F(f)\|.$  To see the other inequality note that for $g \in I_F,$ and $G
  \subseteq X$ a finite set, we
  have $\|f+g\|_L = \sup_G \|\pi_G(f+g)\| \ge \|\pi_F(f+g)\| =
  \|\pi_F(f)\|.$  Hence, $\|\tilde{\pi}_F(f)\| \ge \|\pi_F(f)\|,$  
  and equality follows.

  A similar calculation shows that $\|(\tilde{\pi_F}(f_{i,j}))\| =
  \|(\pi_F(f_{i,j}))\|,$ for any matrix of functions.

  Now it easily follows that $\cl A_L$ is local, since 
  \[ \sup_F
  \|(\tilde{\pi}_F(f_{i,j}))\| = \sup_F \|(\pi_F(f_{i,j}))\| =
  \|(f_{i,j})\|_{M_n(\cl A_L)}. \]
  \end{proof}
  We let {\bf $\tilde{\cl A}$} denote the set of
  functions that are BPW limits of bounded nets of functions in $\cl A.$

  \begin{defn}\label{bpwdefn} 
  Given an operator algebra of functions $\cl A$ on $X$ we say that $f:X
  \to \bb C$ is a {\bf BPW limit} of $\cl A$ if there exists a uniformly
  bounded net $(f_{\lambda})_{\lambda} \in \cl A$ that converges
  pointwise on X to $f.$ We say that $\cl A$ is {\bf BPW complete} if
  it contains the set of functions that are BPW limits of bounded nets
  of functions in $\cl A.$ Let {\bf $\tilde{\cl A}$} denote the set of such functions, that is,
  $\tilde{\cl A}=\{f : \exists \text{ a bounded net } \left<f_{\lambda}\right> \subseteq \cl A \text{ such that } f(z)=lim_{\lambda}f_{\lambda}(z) \ \forall \ z \in X \}$

  Given  $(f_{i,j}) \in M_n(\tilde{\cl A}),$ we set
  \[ \|(f_{i,j})\|_{M_n(\tilde{\cl A})} = inf \{ C: (f_{i,j}(x)) = \lim_{\lambda} (f_{i,j}^{\lambda}(x))
  \text{ and } \|(f_{i,j}^{\lambda})\| \le C \} \]
  \end{defn}

  It is easily checked that for each $n,$ the above formula defines a
  norm on $M_n(\tilde{\cl A}).$ 
  It is also easily checked that a matrix-valued function, $(f_{i,j}): X \to
  M_n$ is the pointwise limit of a uniformly bounded net
  $(f_{i,j}^{\lambda}) \in M_n(\cl A)$ if and only if $f_{i,j} \in
  \tilde{\cl A}$ for every $i,j.$

  \begin{lemma}\label{normform} Let $\cl A$ be an operator algebra of functions on $X$
  and let $(f_{i,j}) \in M_n(\tilde{\cl A}).$  Then
  \begin{multline*} \|(f_{i,j})\|_{M_n(\tilde{\cl A})} = 
  \\ inf\{C: \forall F \subseteq X \text{ finite } \exists \ g_{i,j}^F
  \in \cl A \text{ with } (f_{i,j}\mid_F)=(g_{i,j}^F\mid_F),
  \|(g_{i,j}^F)\| \le C \}.
  \end{multline*}
  \end{lemma}
  \begin{proof} The collection of finite subsets of $X$ determines a
  directed set, ordered by inclusion. If we choose for each finite set $F,$ functions
  $(g_{i,j}^F)$ satisfying the conditions of the right hand set, then
  these functions define a net that converges BPW to $(f_{i,j})$ and
  hence, the right hand side is larger than the left.  Conversely,
  given a net $(f_{i,j}^{\lambda})$ that converges pointwise to
  $(f_{i,j})$ and satisfies $\|(f_{i,j}^{\lambda})\| \le C$ and 
  any finite set $F= \{ x_1,...,x_k \},$ choose
  functions in $\cl A$ such that $f_i(x_j) = \delta_{i,j}.$ If we let
  $A_l^{\lambda} = (f_{i,j}(x_l)) - (f_{i,j}^{\lambda}(x_l)),$ then
  $(g_{i,j}^{\lambda}) = (f_{i,j}^{\lambda}) + A_1^{\lambda} f_1 + \cdots A_k^{\lambda}
  f_k \in M_n(\cl A)$ and is equal to $(f_{i,j})$ on $F.$ Moreover,
  $\|(g_{i,j}^{\lambda})\| \le \|(f_{i,j}^{\lambda})\| +
  \|A^{\lambda}_1\|_{M_n} \|f_1\|_{\cl A} + \cdots +
  \|A_k^{\lambda}\|_{M_n} \|f_k\|_{\cl A}.$ Thus, given $\epsilon >0,$
  since the functions $f_1,...,f_k$ depend only on $F,$
  we may choose $\lambda$ so that $\|(g_{i,j}^{\lambda})\| < C
  +\epsilon.$
  This shows the other inequality.
  \end{proof}
  \begin{lemma}
  Let $\cl A$ be an operator algebra of functions on the set X, then
  $\tilde{\cl A}$ equipped with the collection of norms on
  $M_n(\tilde{\cl A})$ given in
  Definition~\ref{bpwdefn} is an operator algebra.
  \end{lemma}
  \begin{proof} 
  It is clear from the definition of $\tilde{\cl A}$ that it is an algebra.
  Thus, it is enough to check that the axioms of BRS are satisfied by the 
  algebra $\tilde{\cl A}$ equipped with the matrix norms given in 
  the Definition~\ref{bpwdefn}.\\
  \indent If L and M are scalar matrices of appropriate sizes and
  $G\in M_n(\tilde{\cl A}),$ then for $\epsilon > 0$ there exists 
  $G_{\lambda} \in M_n(\cl A)$ such that $\lim_{\lambda}G_{\lambda}(x)=G(x) \ \forall \ x\in X$ and
  $\sup_{\lambda}\|G_{\lambda}\|_{M_n(\cl A)} \leq \|G\|_{M_n(\tilde{\cl A})}+\epsilon.$ 
  Since $\cl A$ is an operator space, $LG_{\lambda}M \in M_n(\cl A)$ 
  and $\|LG_{\lambda} M\|_{M_n(\cl A)}\leq \|L\|\|G_{\lambda}\|_{M_n(\cl A)}\|M\|.$
  Note that it follows that $\|LGM\|_{M_n(\tilde{\cl A})} \leq \|L\|\|G\|_{M_n(\tilde{\cl A})}\|M\|,$ since
  $LG_{\lambda}M \rightarrow LGM $ pointwise and $\sup_{\lambda}
  \|LG_{\lambda}M\|_{M_n(\cl A)} \leq \|L\|(\|G\|_{M_n(\tilde{\cl A})}+\epsilon)\|M\| \ 
  \forall \ \epsilon > 0.$\\
  
  \indent  If $ G, \ H \in  M_n(\tilde{\cl A}),$ then for every
  $\epsilon > 0$ there exists $G_{\lambda}, \ H_{\lambda} \in M_n(\cl A)
  \text{ such} \\ \text {that }  \lim_{\lambda} G_{\lambda}(x)=\lim_{\lambda} G(x) $
  and  $ \lim_{\lambda} H_{\lambda}(x)=H(x)\ \forall  x \in X.$
  Also, we have that $\sup_{\lambda}\|G_{\lambda}\|_{M_n(\cl A)} \leq \|G\|_{M_n(\tilde{\cl A})}+
  \epsilon$ and $\sup_{\lambda}\|H^{\lambda}\|_{M_n(\cl A)} \leq \|H\|_{M_n(\tilde{\cl A})}+ \epsilon.$ \\
  Let $L =GH$ and $L_{\lambda}=G_{\lambda}H_{\lambda}.$ Since ${\cl A}$ is matrix normed algebra, 
  $L_{\lambda} \in M_n(\cl A)$ and $\|L_{\lambda}\|_{M_n(\cl A)} \leq
  \|G_{\lambda}\|_{M_n(\cl A)}\|H_{\lambda}\|_{M_n(\cl A)} $ for every 
  $\lambda.$\\$\Rightarrow  \lim_{\lambda} L_{\lambda}(x)=L(x) \text{ and } 
  \sup_{\lambda}\|L_{\lambda}\|_{M_n(\cl A)} \leq
  \sup_{\lambda}\|G_{\lambda}\|_{M_n(\cl A)} \sup_{\lambda}\|H_{\lambda}\|_{M_n(\cl A)}.$\\ This yields
  $\|L\|_{M_n(\tilde{\cl A})} \leq \|G\|_{M_n(\tilde{\cl A})}\|H\|_{M_n(\tilde{\cl A})},$ 
  and so the multiplication is completely contractive.\\
  \indent Finally, to see that the $L^{\infty}$  conditions are met,
  let $G \in M_n(\tilde{\cl A})$ and $H \in M_m(\tilde{\cl A}).$ 
  Given $\epsilon >0, \  \exists \ G_{\lambda} \in M_n(\cl A)$ and 
  $H_{\lambda} \in M_m(\cl A)$ such that $\lim_{\lambda} G_{\lambda}(x)=G(x) \ , \ \lim_{\lambda}H^{\lambda}(x)=H(x)$ 
  and $ \sup_{\lambda}\|G_{\lambda}\|_{M_n(\cl A)}\leq \|G\|_{M_n(\tilde{\cl A})} +\epsilon, \ 
  \sup_{\lambda}\|H_{\lambda}\|_{M_n(\cl A)} \leq \|H\|_{M_n(\tilde{\cl A})}+\epsilon.$\\ 
  Note that 
  $ G_{\lambda} \oplus H_{\lambda}  \in M_{n+m}(\cl A) $ and
  $ \|G_{\lambda} \oplus H_{\lambda}\| = \max\{\|G_{\lambda}\|_{M_n(\cl A)}
  ,\|H_{\lambda}\|_{M_n(\cl A)} \}$ for every $\lambda$  which implies that 
  $G \oplus H  \in M_{n+m}(\tilde{\cl A}),$ and
  \begin{multline*}
  \|G \oplus H \|_{M_{n+m}(\tilde{\cl A})} \leq  \sup_{\lambda}\| G_{\lambda} \oplus H_{\lambda} \| = 
  \sup_{\lambda}[\max\{ \|G_{\lambda}\|_{M_n(\cl A)},\|H_{\lambda}\|_{M_m(\cl A)}\}] \\ \ \ \ \ \ \ \ =
  \max\{ \sup_{\lambda}\|G_{\lambda}\|_{M_n(\cl A)},\sup_{\lambda}\|H^{\lambda}\|_{M_m(\cl A)}\}
  \\ \leq \max\{ \|G\|_{M_n(\tilde{\cl A})}+\epsilon,\|H\|_{M_m(\tilde{\cl A})} +\epsilon\}\  \ \forall \ \epsilon > 0.
  \end{multline*}
  This shows that 
  $ \| G \oplus H\|_{M_{n+m}(\tilde{\cl A})} \leq \max\{ \|G\|_{M_n(\tilde{\cl A})},
  \|H\|_{M_m(\tilde{\cl A})}\},$
  and so the $L^{\infty}$ condition follows. This completes the proof of the result.
  \end{proof}

  \begin{lemma} If $\cl A$ is an operator algebra of functions on the
  set X, then $\tilde{\cl A} $ equipped with the norms of
  Definition~\ref{bpwdefn} is a local operator algebra of functions on
  X. Moreover, for $(f_{i,j}) \in M_n(\cl A),
  \|(f_{i,j})\|_{M_n(\tilde{\cl A})} = \|(f_{i,j})\|_{M_n(\cl A_L)}.$ 
  \end{lemma}

  \begin{proof}
  It is clear from the definition of the norms on $\tilde{\cl A}$ that the identity 
  map from $\cl A$ to $\tilde{\cl A}$ is completely contractive and thus $\cl A 
  \subseteq \tilde{\cl A}$ as sets. This indeed shows that $\tilde{\cl A}$ separates points
  of X and contains constant functions.

  Let $(f_{ij}) \in M_n(\tilde{\cl A})$ and $\epsilon > 0,$ then $ \exists$ a 
  net $(f_{ij}^{\lambda}) \in M_n(\cl A)$ such that $\lim_{\lambda}(f_{ij}^{\lambda}(x))
  =(f_{ij}(x)) \ \forall \ x\in X$ and $\sup_{\lambda}\|(f_{ij}^{\lambda})\|_{M_n(\cl A)} 
  \leq \|(f_{ij})\|_{M_n(\tilde{\cl A})} +\epsilon.$\\ Since $\cl A$ is an operator algebra
  of functions on the set X, we have that $\|(f_{ij}^{\lambda})\|_{\infty} \leq 
  \|(f_{ij}^{\lambda})\|_{M_n(\cl A)} \ \forall \  \lambda \Longrightarrow 
  \sup_{\lambda}\|(f_{ij}^{\lambda})\|_{\infty} \leq \|(f_{ij})\|_{M_n(\tilde{\cl A})}+\epsilon.$\\
  Fix $ z \in X$, then
  \begin{eqnarray*} 
  \|(f_{ij}(z))\|= \lim_{\lambda}\|(f_{ij}^{\lambda}(z))\|  & \leq &  \sup_{\lambda} 
  \|(f_{ij}^{\lambda})\|_{\infty} \\ &\leq & \|(f_{ij})\|_{M_n(\tilde{\cl A})}+\epsilon
  \ \  \ \ \ \  \forall \ \epsilon > 0 . 
  \end{eqnarray*}
  By letting $\epsilon \rightarrow 0$ and supping over z $\in X, $ we get 
  that $ \|(f_{ij})\|_{\infty} \leq \|(f_{ij})\|_{M_n(\tilde{\cl A})}.$
  Hence, $\tilde{\cl A}$ is an operator algebra of functions on the 
  set X.\\ 
  \indent Denote $\tilde{I}_F =\{f \in \tilde{\cl A}: f|_F \equiv 0\}$
  and let $(f_{ij}) \in M_n(\tilde{\cl A}).$ Then, clearly 
  $\sup_F\|(f_{ij}+\tilde{I}_F)\|_{M_n({\tilde{\cl A}}/{\tilde{I}_F})}
  = \|(f_{ij})\|_{M_n(\tilde{\cl A})}.$ To see the other inequality,
  assume that $\sup_F\|(f_{ij}+\tilde {I}_F)\| < 1,$ then $\forall$ finite
  $ F \subseteq X \ \exists \  (h_{ij}^F) \in M_n(\tilde{\cl A})$
  such that $(h_{ij}^F)|_{F} =(f_{ij}^F)|_F$ and $\sup_F\|h_{ij}^F\| \leq 1.$ 
  Fix a set $F\subseteq X$ and $(h_{ij}^F)\in M_n(\tilde{\cl A})$ then 
  $\forall$ finite $ F'\subseteq X \ \exists \  (k_{ij}^{F'}) \in M_n(\cl A)$
  such that $(k_{ij}^{F'})|_{F'} =(h_{ij}^F)|_{F'}$ and
  $\sup_{F'}\|k_{ij}^{F'}\| \leq 1.$ \\ In particular, let $F'=F$ then 
  $(k_{ij}^{F})|_{F} = (h_{ij}^F)|_{F}=(f_{ij})|_F$ and $\sup_F \|k_{ij}^F\| 
  \leq 1.\\ \ \Rightarrow \ \|(f_{ij})\|_{M_n(\tilde{\cl A})} \leq 1,$  and
  hence $\|(f_{ij})\|_{M_n(\tilde{\cl A})} \leq \sup_F \|(f_{ij}+\tilde{I}_F)\|
  _{M_n({\tilde{\cl A}}/{\tilde{I}_F})}. $\\
  \indent Finally, given that $(f_{ij}) \in M_n(\cl A),$ 
  $\|(f_{ij})\|_{M_n(\tilde{\cl A})}= 
  \sup_F\|(f_{ij}+\tilde{I}_F)\|_{M_n({\tilde{\cl A}}/{\tilde{I}_F})}.$
  Note that  $\forall \ F \subseteq X$ we have $\|(f_{ij}+\tilde{I}_F)\|_{M_n({\tilde{\cl A}}/{\tilde{I}_F})}
  \leq \|(f_{ij}+I_F)\|_{M_n({\cl A}/{I_F})},$ since $I_F \subseteq \tilde{I}_F.$ 
  We claim that for any $(f_{ij}) \in M_n(\cl A),$ and for any finite subset
  $ F\subseteq X,$ we have that $\|(f_{ij}+I_F)\|= \|(f_{ij}+\tilde{I}_F)\|.$
  To see the other inequality, let $(g_{ij}) \in M_n(\tilde{I}_F). $ Then  $\forall \ \epsilon > 0,$ 
  and $ G \subseteq X, \ \exists \  (h_{ij}^G) \in M_n(\cl A)$ such that
  $(h_{ij}^G)|_G= (f_{ij}+g_{ij})|_G$ and $\sup_G\|(h_{ij}^G)\| \leq \|(f_{ij}+g_{ij})\|+ \epsilon.$ 
  Hence, $\|(f_{ij}+I_F)\|= \|(h_{ij}^F+I_F)\| \leq \|(h_{ij}^F)\| \leq \|(f_{ij}+g_{ij})\| +\epsilon \ 
  \ \forall \ \epsilon > 0,$ and the equality follows. Now it is immediate to see that,
  \begin{eqnarray*}
  \|(f_{ij})\|_{M_n(\tilde{\cl A})}= \sup_F \|(f_{ij}+I_F)\|= \|(f_{ij})\|_{M_n(\cl A_L)}
  \end{eqnarray*}
  \end{proof}

   \begin{cor}
   If $\cl A$ is a BPW complete operator algebra then $\cl A_L= \tilde{\cl A}$ completely isometrically.
   \end{cor}
  \begin{proof}  Since $\cl A$ is  BPW complete, $\cl A = \tilde{\cl A}$
  as sets. But by Lemma~\ref{normform}, the norm defined on $\cl A_L$
  agrees with the norm defined on $\tilde{\cl A}.$
  \end{proof}

  \begin{remark}
  In the view of above corollary, we denote the the norm on $\tilde{\cl A}$ by $\|.\|_L.$
  \end{remark}

  \begin{lemma}\label{local}
  If $\cl A$ is an operator algebra of functions on X, then $Ball(\cl A_L)$ is
  BPW dense in $Ball(\tilde{\cl A})$ and $\tilde{\cl A}$ is BPW complete, i.e.,
  $\tilde{\tilde{\cl A}}=\tilde{\cl A}.$
  \end{lemma}
  \begin{proof} It can be easily checked that above is equivalent 
  to showing that $\cl A_L$ is BPW dense in $\tilde{\cl A}.$
  We'll only prove that $\overline{\cl A_L}^{BPW} \subseteq \tilde{\cl A},$
  since the other containment follows immediately by the definition of $\tilde{\cl A}$. \\
  \indent Let $\{f_{\lambda}\}$ be a net in $\cl A_L$ such that
  $f_{\lambda}\rightarrow f$ pointwise and $\sup_{\lambda}\|f_{\lambda}\|_{\cl A_L} < C.$
  Then for fixed $F \subseteq X$ and $\epsilon > 0,\ \exists\  \lambda_F $ such
  that $|f_{\lambda_F}(z) -f(z)| < \epsilon \ \forall z \in F.$ Also 
  since $\sup_{\lambda} \|f_{\lambda}\| < C,$ there exists 
  $g_{\lambda_F} \in I_F$ such that $\|f_{\lambda_F}+g_{\lambda_F}\| < C.$ 
  Note that the function $h_F= f_{\lambda_F}+g_{\lambda_F} \in \cl A$ 
  such that $\|h_F\|_{\cl A} < C,$ and $h_F \rightarrow f$ pointwise. 
  Hence, $f \in \tilde{\cl A}.$ Thus, $\cl A_L$ is BPW dense in $\tilde A.$\\
  \indent Finally, note that the argument similar to the above readily yields that 
  $\tilde{\cl A}$ is BPW complete.
  \end{proof}

  All the above lemmas can be summarized as the following theorem.
 
  \begin{thm}
  If $\cl A$ is an operator algebra of functions on $X,$ then
  $\tilde{\cl A}$ is a BPW complete local operator algebra of functions
  on $X$ which contains $\cl A_L$ completely isometrically as a BPW dense subalgebra.
  \end{thm}

  \begin{defn} Given an operator algebra of functions $\cl A$ on $X,$ we call $\tilde{\cl A}$ 
  the {\bf BPW completion of $\cl A.$}
  \end{defn}

  We now present a few examples to illustrate these concepts. We will delay the main 
  family of examples to the next section.

  \begin{ex} If $\cl A$ is a {\em uniform algebra,} then there exists a compact,
  Hausdorff space $X,$ such that $\cl A$ can be represented as a subalgebra 
  of $C(X)$ that separates points.  If we endow $\cl A$ with the matrix-normed
  structure that it inherits as a subalgebra of $C(X),$  namely,
  $\|(f_{i,j})\| = \|(f_{i,j})\|_{\infty} \equiv sup \{ \|(f_{i,j})\|_{M_n} : x \in X \},$
  then $\cl A$ is a local operator algebra of functions on $X.$  Indeed, to 
  achieve the norm, it is sufficient to take the supremum over all finite 
  subsets consisting of one point.  In this case the BPW completion 
  $\tilde{\cl A}$ is completely isometrically isomorphic to the subalgebra
  of $\ell^{\infty}(X)$ consisting of functions that are bounded, 
  pointwise limits of functions in $\cl A.$
  \end{ex}
  \begin{ex} Let $\cl A = A(\bb D)\subseteq C(\bb D^-)$ be the
  subalgebra of the algebra of continuous functions on the
  closed disk consisting of the functions that are analytic on the
  open disk $\bb D.$ Identifying $M_n(A(\bb D)) \subseteq M_n(C(\bb
  D^-))$ as a subalgebra of the algebra of continuous functions from
  the closed disk to the matrices, equipped with the supremum norm,
  gives $A(\bb D)$ it's usual operator algebra structure. With this
  structure it can be regarded as a local operator algebra of functions 
  on $\bb D$  or on $\bb D^-.$ If we regard it as a local operator algebra
  of functions on $\bb D^-,$ then $A(\bb D) \subsetneq \widetilde{A(\bb
   D)}.$ To see that the containment is strict, note that $f(z) =
  (1+z)/2 \in A(\bb D)$ and $f^n(z) \to \chi_{\{1\}},$ the
  characteristic function of the singleton $\{ 1\}.$

  However, if we regard $A(\bb D)$ as a local operator algebra of functions on $\bb D,$ then
  its BPW completion $\widetilde{A(\bb D)} = H^{\infty}(\bb D),$ the
  bounded analytic functions on the disk, with its usual operator
  structure.
  \end{ex}

  \begin{ex}
  Let $X=\epsilon{\bb D}, \ 0 < \epsilon < 1$ and 
  $\cl A=\{f \in H^{\infty}(\bb D): f:X \rightarrow \bb C\}.$
  If we endow $\cl A$ with the matrix-normed structure 
  on $H^{\infty}(\bb D),$ then $\cl A$ is an operator algebra 
  of functions on X. Also, it can be verified that $\cl A$ is a 
  local operator algebra of functions and that $\cl A = \tilde{\cl A}$. Indeed, if 
  $F=(f_{ij}) \in M_n(\cl A)$ with $\|(f_{ij}+I_Y)\|_{\infty} < 1
  \ \forall$ finite subset $Y\subseteq X,$ then $\exists \ 
  H_Y \in M_n(\cl A)$ such that $\|H_Y\|_{\infty} \leq 1$
  and $H_Y \rightarrow F$ pointwise on X. Note by Montel's 
  theorem $\exists$ a subnet $H_{Y'}$ and $ G \in M_n(H^{\infty}(\bb D))$
  such that $\|G\|_{\infty} \leq 1$ and $H_{Y'} \rightarrow G$ uniformly 
  on compact subsets of $\bb D.$ Thus, by the identity 
  theorem $F \equiv G$ on $\bb D.$ Hence, $\cl A$ is a 
  local operator algebra. Finally, by using 
  Lemma~\ref{local} and a similar argument, one 
  can also show that $\tilde{\cl A}= H^{\infty}(\bb D).$
  \end{ex}

  \begin{ex}
  Let $\cl A= H^{\infty}(\bb D)$ but endowed with a new norm. 
  Fix $b > 1,$ and set $\|F\|=\max\{\|F\|_{\infty}, \|F(\begin{pmatrix}
  0 & b\\0 & 0\end{pmatrix})\|\}, \ F\in M_n(\cl A).$ It can 
  be easily verified that $\cl A$ is a BPW complete operator
  algebra of functions. However, we also claim that $\cl A$ 
  is local. To prove this we proceed by contradiction. Suppose $\exists \ F=(f_{ij}) \in M_n(H^{\infty}(\bb D))$ such that 
  $\|F\| > 1 > c, $ where $c=  \sup_Y\|(f_{ij}+I_Y)\|$.\\
  In this case, 
  $\|F\|=\|F(\begin{pmatrix} 0 & b\\0 & 0\end{pmatrix})\|,$ since
  $\|(f_{ij}+I_Y)\|=\|F(\lambda)\|$ when $Y=\{\lambda\} .$
  Let $\epsilon=\frac{1-c}{4b}$ and $Y= \{ 0, \epsilon \} \subseteq \bb D,$ 
  then $ \exists \ G \  \in M_n(H^{\infty}(\bb D))$
  such that $G|_Y=0$ and $\|F+G\| < \frac{1+c}{2}.$ 
  Thus, we can write $B_Y(z)={\frac{z-\epsilon}{1-\bar{\epsilon}z}},$
  so that we can write $G(z)= zB_Y(z)H(z),$
  for some $ H \in M_n( H^{\infty}).$ It follows that $\|H\|_{\infty} <2,$ 
  since $\|G\|_{\infty} < 2.$ We now consider
  \begin{multline*}
  1  < \|F(\begin{pmatrix} 0 & b\\0 & 0\end{pmatrix})\| 
  \leq \|(F+G)(\begin{pmatrix} 0 & b\\0 & 0\end{pmatrix})\| +
  \|G(\begin{pmatrix} 0 & b\\0 & 0\end{pmatrix})\| \\
  \leq \frac{1+c}{2}+\|\begin{pmatrix} 0 & b G'(0)\\ 0 & 0 \end{pmatrix}\| 
  = \frac{1+c}{2} +b |B_Y(0)|\|H(0)\|\\ \leq  \frac{1+c}{2} +2b \epsilon 
  = \frac{1+c}{2} + 2b \frac{1-c}{4b}= 1 \text{ contradiction.}
  \end{multline*}
  \end{ex}


   \begin{ex}
   This is an example of a non-local algebra that arises from boundary
   behavior.
   Suppose $\cl A=\{ f \in A(\bb D)\subseteq C(\bb D^-): f: \bb D \rightarrow \bb C\}$
   equipped with the family of matrix norms $\|F\|=\max\{\|F\|_{\infty}, \|F(\begin{pmatrix}
   1 & 1\\0 & -1 \end{pmatrix})\|\}, F \in M_n(\cl A).$ Then it is easy to check that
   $\cl A $ is an operator algebra of functions on the set $\bb D.$ Also, it can be 
   verified that $\cl A$ is not local. To see this, note that $\|z\|= \|\begin{pmatrix}
   1 & 1\\0 & -1 \end{pmatrix}\|.$  For each $ Y=\{z_1,z_2,\ldots, z_n\},$
   we define $B_Y(z)=\Pi_{i=1}^{n}(\frac{z-z_i}{1-\bar{z_i}z})$ and choose 
   $h \in \cl A$ such that $h(1)=-\overline{B_Y(1)}$ and $h(-1)=\overline{B_Y(-1)}.$\\
   Let  $g(z)= z+ B_Y(z)h(z)\alpha ,$ where $\alpha =\frac{\|z\|-1}{\|h\|_{\infty}+\|z\|} > 0.$
   Then $ g \in \cl A$ and $g|_Y=z|_Y \Longrightarrow \|\pi_Y(z)\| =
   \|\pi_Y(g)\|\leq \|g\|\leq \max\{ 1+\alpha \|h\|_{\infty}, (1-\alpha)\|z\|\} < \|z\|.$\\
   Thus, $\sup_{Y\subseteq \bb D}\|\pi_Y(z)\| < \|z\|$ and hence $\cl A$
   is not local.
   \end{ex}

   \begin{ex} This example shows that one can easily build non-local
     algebras by adding ``values'' outside of the set X.
   Let $\cl A$ be the algebra of polynomials regarded as functions
   on the set $X=\bb D.$ Then $\cl A$ endowed with the matrix-normed
   structure as $\|(p_{ij})\|=\max\{\|(p_{ij})\|_{\infty}, \|(p_{ij}(2))\|\},$
   is an operator algebra of functions on the set X. To see that $\cl A$ is not local, let $p \in \cl A$ be
   such that $\|p\|_{\infty} < |p(2)|.$ For each finite subset $Y=\{z_1,\ldots,z_n\}$ of X,
   let $h_Y(z)= \Pi_{i=1}^{n}(z-z_i)$ and $g_Y(z)= p(z)- \alpha h_Y(z)p(2),$ where 
   $\alpha=\frac{|p(2)|-\|p\|_{\infty}}{2|p(2)|\|h_Y\|_{\infty}}>0.$
   Note that $\|g_Y\| \leq (1-\alpha)|p(2)|$ and
   $g_Y|_Y=p|_Y \Longrightarrow \|\pi_Y(p)\|=\|\pi_Y(g_Y)\| \leq
   \|g_Y\|\le (1- \alpha)|p(2)| < \|p\|.$
   Hence, it follows that $\cl A$ is not local.\\
   Finally, observe that in this case $\cl A$ cannot be BPW complete.
   For example, if we take $p_n= \frac{1}{3}\sum_{i=0}^{n}(\frac{z}{3})^i \in \cl A$ 
   then $p_n(z) \rightarrow f(z)=\frac{1}{3-z}  \ \forall \ z \in \bb D$
   and $\|p_n\| < \|f\| \Rightarrow \cl A_L \subsetneq \tilde{A}.$
   \end{ex}

   \begin{ex} It is still an open problem as to whether or not every
   unital contractive, homomorphism $\rho: H^{\infty}(\bb D) \to B(\cl
   H)$ is completely contractive. For a recent discussion of this
   problem see \cite{PR}. Let's assume that $\rho$ is a
   contractive homomorphism that is not completely contractive.
   Let $\cl B = H^{\infty}(\bb D),$ but endow it with the family of
   matrix-norms given by,
   \[ |||(f_{i,j})||| = max \{ \|(f_{i,j})\|_{\infty},
   \|(\rho(f_{i,j}))\| \}. \]
   Note that $|||f||| = \|f\|_{\infty},$ for $f \in \cl B.$

   It is easily checked that $\cl B$ is a BPW complete operator algebra of functions
   on $\bb D.$  However, since every contactive homomorphism of $A(\bb
   D)$ is completely contractive, we have that for $(f_{i,j}) \in
   M_n(A(\bb D)),  |||(f_{i,j})||| = \|(f_{i,j})\|_{\infty}.$
   If $Y=\{x_1,...,x_k \}$ is a finite subset of $\bb D$ and $F=(f_{i,j}) \in M_n(\cl B),$ 
   then there is $G=(g_{i,j}) \in M_n(A(\bb D)),$ such that $F(x) = G(x)$ for all $x \in Y,$ 
   and $\|G\|_{\infty} = \|F\|_{\infty}.$ Hence, $\|\pi_Y^{(n)}(F)\| \le \|F\|_{\infty}.$ 
   Thus, $\sup_Y \|\pi_Y^{(n)}(F)\| = \|F\|_{\infty}.$
   It follows that $\cl B$ is not local and that $\tilde{\cl B}
   = \cl B_L = H^{\infty}(\bb D),$ with its usual supremum norm operator algebra
   structure.

   In particular, if there does exist a contractive but not completely
   contractive representation of $H^{\infty}(\bb D),$ then we have
   constructed an example of a non-local BPW complete operator algebra of
   functions on $\bb D.$
   \end{ex}

   \section{A Characterization of Local Operator Algebras of Functions}

  The main goal of this section is to prove that every BPW complete local 
  operator algebra of functions is completely isometrically isomorphic to 
  the algebra of multipliers on a reproducing kernel Hilbert space of 
  vector-valued functions.  Moreover, we wil show that every such algebra 
  is a dual operator algebra in the precise sense of \cite{BLM}. We will 
  then prove that for such BPW algebras, weak*-convergence and BPW 
  convergence coincide on bounded balls.

  Given a set $X$ and a Hilbert space $\cl H,$ then by a {\em reproducing kernel Hilbert space of $\cl H$-valued functions,}
  we mean a vector space $\cl L$ of $\cl H$-valued functions that is 
  equipped with a norm and an inner product that makes it a Hilbert space 
  and which has the property that for every $x \in X,$ the evaluation map 
  $E_x: \cl L \to \cl H,$ is a bounded, linear map.
  Recall that given a Hilbert space $\cl H,$ a matrix of operators, 
  $T=(T_{i,j}) \in M_k(B(\cl H))$ is regarded as an operator on the 
  Hilbert space $\cl H^{(k)} \equiv \cl H \otimes \bb C^k,$ which 
  is the direct sum of $k$ copies of $\cl H.$ A function 
  $K:X \times X \to B(\cl H),$ where $H$ is a Hilbert space, 
  is called a {\em positive definite operator-valued function on X,} 
  provided that for every finite set of (distinct)points $\{x_1,...,x_k \}$ in $X$, 
  the operator-valued matrix, $(K(x_i,x_j))$ is positive semidefinite. 
  Given a reproducing kernel Hilbert space of $\cl H$-valued functions, 
  if we set $K(x,y) = E_xE_y^*,$ then $K$ is positive definite and is 
  called the {\em reproducing kernel of $\cl L.$}   There is a converse 
  to this fact, generally called {\em Moore's theorem,} which states 
  that given any positive definite operator-valued function $K:X \times X \to B(\cl H),$ 
  then there exists a unique reproducing kernel Hilbert space of $\cl H$-valued 
  functions on X, such that $K(x,y) = E_xE_y^*.$ We will denote this space by 
  $\cl L(K, \cl H)$.
  Finally, given any reproducing kernel Hilbert space $\cl L$ of $\cl H$-valued 
  functions with reproducing kernel $K,$ a function $f:X \to \bb C$ is called
  a {\em multiplier} provided that for every $g \in \cl L,$ the function 
  $fg \in \cl F.$  In this case it follows by an application of the closed 
  graph theorem that the map $M_f: \cl L \to \cl L,$ defined by $M_f(g) =fg,$ 
  is a bounded, linear map.  The set of all multipliers is denoted $\cl M(K)$ 
  and is easily seen to be an algebra of functions on $X$ and a subalgebra of $B(\cl L).$
  The reader can find proofs of the above facts in \cite{BM} \cite{DAJH}.
  Also, we refer to the fundamental work of Pedrick\cite{Pe} for further treatment
  of vector-valued reproducing kernel Hilbert spaces.

  \begin{lemma}\label{multdual}  Let $\cl L$ be a reproducing kernel Hilbert 
  space of $\cl H$-valued functions with reproducing kernel $K:X \times X \to B(\cl H).$
  Then $\cl M(K) \subseteq B(\cl L)$ is a weak*-closed subalgebra.
  \end{lemma}
  \begin{proof} It is enough to show that the unit ball is weak*-closed by the 
  Krein-Smulian theorem. So let $\{ M_{f_{\lambda}} \}$ be a net of multipliers 
  in the unit ball of $B(\cl L)$ that converges in the weak*-topology to an 
  operator $T.$ We must show that $T$ is a multiplier.

  Let $x \in X$ be fixed and assume that there exists $g \in \cl L,$ with 
  $g(x)=h \ne 0.$ Then $\langle Tg, E_x^*h \rangle_{\cl L} = \lim_{\lambda} \langle M_{f_{\lambda}}g, E_x^*h \rangle_{\cl L} = \lim_{\lambda} \langle        E_x(M_{f_{\lambda}}g), h \rangle_{\cl H} = \lim_{\lambda} f_{\lambda}(x) \|h\|^2.$
  This shows that at every such $x$ the net $\{ f_{\lambda}(x) \}$ converges 
  to some value. Set $f(x)$ equal to this limit and for all other $x$'s set 
  $f(x) = 0.$  We claim that $f$ is a multiplier and that $T=M_f.$

  Note that if $g(x) =0$ for every $g \in \cl L,$ then $E_x=E_x^*=0.$
  Thus, we have that for any $g \in \cl L$ and any $h \in \cl H,$
  $\langle E_x(Tg), h \rangle_{\cl H} = \lim_{\lambda} \langle E_x(M_{f_{\lambda}}g)
  , h \rangle_{\cl H} = \lim_{\lambda} f_{\lambda}(x) \langle g(x), h \rangle_{\cl H} = f(x) \langle g(x) , h \rangle_{\cl H}.$
  Since this holds for every $h \in \cl H,$ we have that 
  $E_x(Tg) = f(x)g(x),$ and so $T = M_f$ and $f$ is a multiplier.
  \end{proof}  

  Every weak*-closed subspace of $V \subseteq B(\cl H)$ has a predual 
  and it is the operator space dual of this predual.
  Also, if an abstract operator algebra is the dual of an operator space,
  then it can be represented completely isometrically and 
  weak*-continuously as a weak*-closed subalgebra of the bounded 
  operators on some Hilbert space.  For this reason an operator algebra 
  that has a predual as an operator space is called a {\em dual operator algebra.}  
  See the book of \cite{BLM} for the proofs of these facts.  
  Thus, in summary, the above lemma shows that every multiplier 
  algebra is a dual operator algebra. 

  \begin{thm} Let $\cl L$ be a reproducing kernel Hilbert space of $\cl H$-valued
  functions with reproducing kernel $K:X \times X \to B(\cl H)$ and let 
  $\cl M(K) \subseteq B(\cl L)$ denote the multiplier algebra, endowed with 
  the operator algebra structure that it inherits as a subalgebra. If $K(x,x) \ne 0,$
  for every $x \in X$ and $\cl M(K)$ separates points on $X,$ then $\cl M(K)$ is a 
  BPW complete local dual operator algebra of functions on $X.$
  \end{thm}
  \begin{proof} The multiplier norm of a given matrix-valued function 
  $F= (f_{i,j})\in M_n(\cl M(K))$ is the least constant $C$ such that
  $((C^2 I_n - F(x_i)F(x_j)^*) \otimes K(x_i,x_j)) \ge 0,$ for all sets 
  of finitely many points, $Y= \{x_1,..., x_k \} \subseteq X.$
  Applying this fact to a set consisting of a single point, we have that 
  $(C^2I_n - F(x)F(x)^*) \otimes K(x,x) \ge 0,$ and it follows that 
  $C^2I_n - F(x)F(x)^* \ge 0.$ Thus, $\|F(x)\| \le C= \|F\|$ and we have 
  that point evaluations are completely contractive on $\cl M(K).$ Since 
  $\cl M(K)$ contains the constants and separates points by hypothesis, 
  it is an operator algebra of functions on $X.$

  Suppose that $\cl M(K)$ was not local, then there would exist 
  $F \in M_n(\cl M(K)),$ and a real number $C,$ such that 
  $\sup_Y \|\pi_Y^{(n)}\| < C < \|F\|.$
  Then for each finite set $Y= \{x_1,...,x_k \}$ we could choose 
  $G \in M_n(\cl M(K)),$ with $\|G\| < C,$ and $G(x) = F(x),$ for every $x \in Y.$
  But then we would have that $((C^2I_n - F(x_i)F(x_j)^*) \otimes K(x_i,x_j)) = ((C^2 I_n - G(x_i)G(x_j)^*) \otimes K(x_i,x_j)) \ge 0,$
  and since $Y$ was arbitrary, $\|F\| \le C,$ a contradiction. 
  Thus, $\cl M(K)$ is local.

  Finally, assume that $f_{\lambda} \in \cl M(K),$ is a net in $\cl M(K),$ 
  with $\|f_{\lambda}\| \le C,$ and $\lim_{\lambda}(x) = f(x),$ pointwise.
  If $g \in \cl L$ with $\|g\|_{\cl L} = M,$ then $(MC)^2K(x,y) - \langle f_{\lambda}(x)g(x),f_{\lambda}(y)g(y) \rangle$ 
  is positive definite. By taking pointwise limits, we obtain that
  $(MC)^2K(x,y) - \langle f(x)g(x), f(y)g(y) \rangle$ is positive definite. 
  From the characterization of functions and their norms in a reproducing 
  kernel Hilbert space, this implies that $fg \in \cl L,$ with 
  $\|fg\|_{\cl L} \le MC.$ Hence, $f \in \cl M(K)$ with $\|M_f\| \le C.$ 
  Thus, $\cl M(K)$ is BPW complete.
  \end{proof}

  In general, $\cl M(K)$ need not separate points on $X.$ In fact, it 
  is possible that $\cl L$ does not separate points and if $g(x_1) =g(x_2),$
  for every $g \in \cl L,$ then necessarily $f(x_1) =f(x_2)$ for every $f \in \cl M(K).$

  Following \cite{Pa2}, by a {\em $k$-idempotent operator algebra,} $\cl C,$ we mean that 
  we are given $k$ operators, $\{ E_1,...,E_k \}$ on some Hilbert space $\cl H,$ 
  such that $E_iE_j = E_jE_i = \delta_{i,j} E_i,$  $I= E_1+ \cdots + E_k$ 
  and $\cl C = span \{ E_1,..., E_k \}.$

  \begin{prop} Let $\cl C= span \{ E_1,..., E_k \}$ be a $k$-idempotent operator
  algebra on the Hilbert space $\cl H,$ let $Y= \{ x_1,...,x_k\}$ be a set 
  of $k$ distinct points and define $K:Y \times Y \to B(\cl H)$ by 
  $K(x_i,x_j) = E_iE_j^*.$ Then $K$ is positive definite and $\cl C$ 
  is completely isometrically isomorphic to $\cl M(K)$ via the map that 
  sends $a_1E_1 + \cdots + a_kE_k$ to the multiplier $f(x_i) = a_i.$
  \end{prop}
  \begin{proof} It is easily checked that $K$ is positive definite. We 
  first prove that the map is an isometry. Given
  $B= \sum_{i=1}^k a_i \otimes E_i \in \cl C,$ let $f:Y \to \bb C$ be 
  defined by $f(x_i) = a_i.$  We have that $f \in \cl M(K)$ with $\|f\| \le C$
  if and only if $P= ((C^2- f(x_i)f(x_j)^*)K(x_i,x_j))$ is positive 
  semidefinite in $B(\cl H^{(k)}).$  
  Let $v = e_1 \otimes v_1 + \cdots e_k \otimes v_k \in \cl H^{(k)},$ 
  let $h = \sum_{j=1}^k E_j^*v_j$ and note that $E_j^*h = E_j^*v_j.$
  Finally, set $h = \sum_{i=1}^k h_i.$
  Thus,  \begin{multline*} \langle Pv,v \rangle =
  \sum_{i,j=1}^k (C^2 - a_i\bar{a_j}) \langle E_iE_j^*v_j,v_i \rangle =
   \\ \sum_{i,j=1}^k (C^2 - a_i \bar{a_j})        \langle E_j^* h, E_i^* h \rangle = C^2\|h\|^2 - \langle B^*h, B^*h \rangle
   = \\  C^2 \|h\|^2 - \|B^*h\|^2. \end{multline*}

  Hence, $\|B\| \le C$ implies that $P$ is positive and so 
  $\|M_f\| \le \|B\|.$  For the converse, given any $h$ let 
  $v = \sum_{j=1}^k e_j \otimes E_j^* h,$ and note that $\langle Pv,v \rangle \ge 0,$
  implies that $\|B^*h\| \le C,$ and so $\|B\| \le \|M_f\|.$

  The proof of the complete isometry is similar but notationally cumbersome.
  \end{proof}

  \begin{thm}\label{tilde-multiplier}
  Let $\cl A $ be an operator algebra of functions on the set X then there
  exists a Hilbert space, $\ff C$  and a positive definite function 
  $K:X \times X \rightarrow B(\ff C)$ such that $\cl M(K)=\tilde{\cl A}$ 
  completely isometrically.
  \end{thm}
  \begin{proof}
  Let Y be a finite subset of X. Since $\cl A/I_Y$ is a $|Y|$-idempotent 
  operator algebra, by the above lemma, there exists a vector valued kernel 
  $K_Y$ such that $\cl A/I_Y= \cl M(K_Y)$ completely isometrically.\\ 
  Define \[\widetilde{K_Y}(x,y)= \left\{\begin{array}{c l}
  K_Y(x,y)& \text{ when } (x,y) \in Y \times Y ,\\
  0& \text{ when }(x,y)\not \in Y \times Y.
  \end{array}
  \right. \]
  and set $K=\sum_Y \oplus \widetilde{K_Y},$ where the direct sum is over
  all finite subsets of X.\\
  Then it is easily checked that K is positive definite.\\
  \noindent  Let $f \in M_n(\cl M(K))$ with $\|M_f\| \leq 1.$ This is equivalent
  to  \begin{multline*} ((I-f(x_i)f(x_j)^*)\otimes K(x_i,x_j)) \geq 0 \\
  \Longleftrightarrow ((I-f(x_i)f(x_j)^*)\otimes K_{F}(x_i,x_j)) \geq 0 \ \forall \ F\subseteq X. \end{multline*}
  This last condition is equivalent to the existence for each $F$ of some
  $f_F \in M_n(\cl A)$ such that $\|\pi_F(f_F)\| \le 1$ and $f_F =f$ on $F.$ 
  The net of functions $\{ f_F \}$ then converges BPW to $f.$ Hence, 
  $f \in \tilde{\cl A}$ with $\|f\|_L \le 1.$
 
  This proves that $M_n(\cl M(K))= M_n(\tilde{\cl A})$ isometrically.
  \end{proof}
  \begin{cor}
  Every BPW complete local operator algebra of functions is a dual operator algebra.
  \end{cor}
  \begin{proof} In this case we have that $\cl A = \tilde{\cl A} = \cl M(K)$ 
  completely isometrically.  By Lemma~\ref{multdual}, this latter algebra is 
  a dual operator algebra.
  \end{proof}

  \noindent The above theorem gives a weak*-topology to a local 
  operator algebra of function, $\cl A$ by using the identification 
  $\cl A \subseteq \tilde{\cl A} = \cl M(K)$ and taking the weak*-topology 
  of $\cl M(K).$ 
The following proposition 
  proves that convergence of bounded nets in this weak*-topology on 
  $\cl A$  is same as BPW convergence.

  \begin{prop}
  Let $\cl A$ be a local operator algebra of functions on the set X. 
  Then the net $(f_{\lambda})_{\lambda} \in Ball(\cl A)$ converges 
  in the weak*-topology if and only if it converges pointwise on X.
  \end{prop} 
  \begin{proof} Let $\cl H$ denote the reproducing kernel Hilbert
    space of $\ff C$-valued functions on X with kernel $K$ for which $\tilde{\cl A} =
    \cl M(K).$ Recall that if $E_x: \cl H \to \ff C,$ is the linear
    map given by evaluation at $x,$ then $K(x,y) = E_x E_y^*.$
Also, if $v \in \ff C,$ and $h \in \cl H,$ then $\langle h, E_x^*v
\rangle_{\cl H} = \langle h(x), v \rangle_{\ff C}.$

  First assume that the net $(f_{\lambda})_{\lambda} \in Ball(\cl A)$ 
  converges to f in the weak*-topology. Using the identification of
  $\tilde{\cl A} = \cl M(K),$ we have that the operators
  $M_{f_{\lambda}}$ of
  multiplication by $f_{\lambda},$ converge in the weak*-topology of
  $B(\ff C)$ to $M_f.$  Then for any $x \in X, h \in \cl H, v \in \ff
  C,$ we have that $f_{\lambda}(x)\langle h(x),v \rangle_{\ff C} =
  \langle f(x)h(x), v \rangle_{\ff C} = \langle M_{f_{\lambda}}h,
  E_x^*v \rangle_{\cl H} \to \langle M_fh, E_x^*v \rangle_{\cl H} =
  f(x) \langle h(x),v \rangle_{\ff C}.$ Thus, if there is a vector in
  $\ff C$ and a vector in $\cl H$ such that $\langle h(x),v
  \rangle_{\ff C} \ne 0,$ then we have that $f_{\lambda}(x) \to f(x).$
  It is readily seen that such vectors exist if and only if $E_x \ne
  0,$ or equivalently, $K(x,x) \ne 0.$  But this follows from the
  construction of K as a direct sum of positive definite functions
  over all finite subsets of X. For fixed $x \in X$ and the one
  element subset 
  $Y_0=\{x\}$, we have that the 1-idempotent algebra $\cl A/I_{Y_0}
  \ne 0$ and so $K_{Y_0}(x,x) \ne 0,$ which is one term in the direct
  sum for $K(x,x).$

  
\indent Conversely, assume that $ \|f_{\lambda}\| < K, \forall \lambda $ and $ f_{\lambda} 
  \rightarrow f$ pointwise on X. We must prove that $M_{f_{\lambda}}
  \to M_f$ in the weak*-topology on $B(\cl H).$ But since this is a
  bounded net of operators, it will be enough to show convergence in the
  weak operator topology and arbitrary vectors can be replaced by
  vectors from a spanning set. Thus, it will be enough to show that
  for $v_1,v_2 \in \ff C$ and $x_1,x_2 \in X,$ we have that
$\langle M_{f_{\lambda}} E_{x_1}^*v_1, E_{x_2}^*v_2 \rangle_{\cl H} \to \langle M_f
E_{x_1}^*v_1,E_{x_2}^*v_2 \rangle_{\cl H}.$ But we have, 
\begin{multline*} \langle M_{f_{\lambda}} E_{x_1}^*v_1, E_{x_2}^*v_2 \rangle_{\cl H}=
\langle E_{x_2}(M_{f_{\lambda}}E_{x_1}^*v_1),v_2 \rangle_{\ff C}=\\
f_{\lambda}(x_2)\langle K(x_2,x_1)v_1,v_2 \rangle_{\ff C} \to f(x_2)
\langle K(x_2,x_1)v_1,v_2 \rangle_{\ff C} = \langle M_f
E_{x_1}^*v_1,E_{x_2}^8v_2 \rangle_{\cl H}, \end{multline*} 
and the result follows.
  \end{proof}

  \begin{cor}
  The ball of a local operator algebra of functions is weak*-dense in the ball of its BPW completion.
  \end{cor}


  \section{Residually Finite Dimensional Operator Algebras}

   A C*-algebra $B$ is called {\bf residually finite-dimensional
  (RFD)} if it has a separating family of finite-dimensional
  representations. Note that when $B$ is separable and RFD
  it has a seperating {\em sequence} of finite dimensional
  representations.\\ Since every one-to-one $*$-homomorphism
  of a C*-algebra is completely isometric, a C*-algebra $B$
  is RFD if and only if for all $n,$ and for every $(b_{i,j})
  \in M_n(B),$ we have that $\|(b_{i,j})\|_{M_n(B)}=
  \sup  \{\|(\pi(b_{i,j}))\|\}$ where the suprema is taken over all
  $*$-homomorphisms, $\pi:B \rightarrow \bb M_k \ \forall \  k.$ 
  RFD C*-algebras have been studied in \cite{BB},\cite{DM},\cite{RJ},\cite {GM}.\\

  We know from BRS that any operator algebra can
  be represented on a Hilbert space, i.e., given an
  abstract unital operator algebra $\cl A$ there
  exists an unital completely isometric homomorphism
  $\pi:\cl A \rightarrow \bb B(\cl H).$ Since a representation of a
  C*-algebra is a *-homomorphism if and only if it is completely
  contractive, the following
  definition gives us a natural way of extending the
  notion of RFD to operator algebras.

  \begin{defn}
  An operator algebra, $\cl B$ is called {\bf RFD} if for every $n$ and
  for every
  $(b_{i,j}) \in M_n(\cl B),$ $ \|(b_{i,j})\|=\sup \{\|(\pi(b_{i,j}))\|\},$ 
  where the suprema is taken over all completely contractive homomorphisms, 
  $\pi:\cl B \rightarrow M_k \ \forall \ k.$
  A dual operator algebra $\cl B$ is called {\bf weak*-RFD} if this last
  equality holds when the completely contractive homomorphisms are
  also required to be weak*-continuous. 
  \end{defn}

  The following result is implicitly contained in \cite{Pa2}, but the
  precise statement that we shall need does not appear there.
  Thus, we refer the reader to \cite{Pa2} to be able to fully understand
  the proof since we have used some of the definitions
  and results from \cite{Pa2} without stating them.
 
  \begin{lemma}
  Let $\cl B$ be a concrete k-idempotent operator algebra. Then
  $\ff S(\cl B^*\cl B))$ is a Schur ideal affiliated with $\cl B,$
  i.e., $\cl B=\ff A(\ff S(\cl B^*\cl B))$ completely isometrically.
  \end{lemma}
  \begin{proof} From the Corollary~3.3 of \cite{Pa2} we have that
   the Schur ideal $\ff S(\cl B^*\cl B)$ is non-trivial and bounded.
   Thus, we can define the algebra $\ff A(\ff S(\cl B^*\cl B))=span\{E_1,\cdots, E_k\},$ 
   where $E_i=\sum_n \sum_{ Q \in\ff S_n^{-1}(\bb B^*\cl B)}\oplus Q^{1/2}(I_n \otimes E_{ii})Q^{-1/2}$
   is the idempotent operator that live on $\sum_n\sum_{Q \in \ff S_n^{-1}}\oplus M_k(M_n).$
   By using Theorem 3.2 of \cite{Pa2} we get that 
   $\ff S(\ff A(\ff S(\cl B^*\cl B))\ff A(\ff S(\cl B^*\cl B))^*)=\ff S(\cl B^*\cl B)$
   This further imples that  $\ff A(\ff S(\cl B^*\cl B))\ff A(\ff S(\cl B^*\cl B))^*= \cl B^*\cl B$
   completely order isomorphically under the map which sends 
   $E_i^*E_j$ to $F_i^*F_j.$ Finally, by restricting the same map
   to A we get a map which sends $E_i$ to $F_i$ completely isometrically.
   Hence, the result follows.
  \end{proof}


  \begin{thm} Every k-idempotent operator algebra is weak*-RFD.
  \end{thm}
  \begin{proof} Let $\cl A$ be an abstract k-idempotent operator algebra.
  Note that $\cl A$ is a dual operator algebra being a finite
  dimensional operator algebra. From this it follows that there exists
  a Hilbert space,\ $\cl H$ and a weak*-continuous completely
  isometric homomorphism, $ \pi:\cl A \rightarrow B(\cl H).$
  Note that $\cl B = \pi(\cl A)$ is a concrete k-idempotent algebra
  generated by the idempotents,  $\cl B=span\{F_1, F_2,..., F_k\}$
  contained in $B(\cl H).$  Thus, from the above lemma 
  $\cl B= \ff A(\ff S(\cl B^* \cl B))$ completely isometrically.

  \noindent For each $n \in \bb N, \ Q \in \ff S_n^{-1}$, we define
   $\pi_n^Q: \cl B \rightarrow M_k(M_n)$ via 
   $$\pi_n^Q(F_i)= Q^{1/2}(I_n\otimes E_{ii})Q^{-1/2}$$
  Assume for the moment that we have proven that $\pi_n^Q$ is a weak*-
  continuous completely contractive homomorphism. Then for every
  $(b_{ij}) \in M_k(\cl B)$ we must have that
  $\ \sup_{n, Q \in \ff S_n^{-1}}\|(\pi_n^Q(b_{ij}))\|=\|(b_{ij})\| ,$
  and hence $\|(b_{ij})\|=\sup_m\{\|(\rho(b_{ij}))\|\}$ where supremum is taken
  over all weak*-continuous cc homomorphism $ \rho:\cl B
  \rightarrow M_m \ \forall \ m.$ \\
  \indent Finally, note that for any $(a_{ij}) \in M_k(\cl A)$
  we have that $\|(a_{ij})\|_{M_k(\cl A)} =  \sup_m\{\|(\rho(\pi(a_{ij})))\|\}$
  where supremum is taken over all weak*-continuous cc homomorphisms $\rho:{\cl B}^*\cl B \rightarrow M_m$
  which follows from the fact that $\pi$ is a complete isometry.
  Since $\rho \circ \pi$ is a weak*-continuous cc homomorphism on $\cl A$
  for every weak*-continuous cc homomorphism, $\rho$ on $\cl B^*\cl B.$
  This implies that $ \|(a_{ij})\|_{M_k(\cl A)}= \sup\{\|(\gamma(a_{ij}))\|\}$
  where supremum is taken over all finite dimensional weak*-continuous
  cc homomorphisms, and hence the result.\\
  \indent Thus, it remains to show that $\pi_n^Q$ is a weak*-continuous
  cc homomorphism. Note that it is easy to check
  that it is a completely contractive homomorphism and
  $\sup_{n, Q\in \ff S_n^{-1}}\|(\pi_n^Q(b_{ij}))\|=\|(b_{ij})\|
  \  \ \forall \ (b_{ij}) \in M_k(\cl B^*\cl B).$ \\
  \indent Finally, to see that it is weak*-continuous, let
  $f_{\lambda}$ be a bounded net in $\cl B^*\cl B$ which converges
  in the weak*-topology. Note that we can write $f_{\lambda}=
  \sum_{i=1}^{k}\alpha_{1}^{k}E_i$ and $f=\sum_{i=1}^{k}\alpha_{1}E_i.$ \\
  We define for each i,  $ \psi_i: \cl B^*\cl B \rightarrow \bb C$ via 
  $\psi_i(g)= \alpha_{i}^{g} $ where
  $g=\sum_{i=1}^{k}\alpha_{i}^{g}E_i.$\\ Note that for each i, 
  $\psi_i$ is a bounded linear functional on $\cl B^*\cl B,$
  since the linear maps on a finite dimensional normed
  space are bounded. Thus, for given $\epsilon > 0, \ \exists
  \  \lambda_0$ such that
  $|\alpha_{i}^{\lambda} - \alpha_i| < \frac{\epsilon}{\max_{1 \leq i\leq k}\{\pi_n^Q(E_i)\}} 
  \ \forall \lambda \geq \lambda_0.$\\ 
  \indent Finally, to see that $\pi_n^Q(f_{\lambda})
  \rightarrow \pi_n^Q(f)$ in the weak*-topology, fix unit vectors
  $h,\  k \in \cl H$ and consider 
  \begin{multline*}
  |\langle \pi_n^Q(f_{\lambda})h,k \rangle - \langle \pi_n^Q(f)h,k \rangle|
  = |\sum_{i=1}^{k} (\alpha_{i}^{\lambda}-\alpha_i) \langle \pi_n^Q(E_i)h,k \rangle|\\  \leq \sum_{i=1}^{k}|\alpha_{i}^{\lambda}-\alpha_i| |\langle
  \pi_n^Q(E_i)h,k \rangle| < \epsilon \ \  \forall \lambda \geq \lambda_0.
  \end{multline*}
  This shows that the map $\pi_n^Q$ is a weak*-continuous and cc homomorphism for every $Q \in \ff S_n^{-1}$ and n.
  \end{proof}

  \begin{thm} Every BPW complete local operator algebra of functions is
   waek*-RFD.
  \end{thm}
  \begin{proof} Let $\cl A$ be a BPW complete local operator algebra of
  functions on the set X and F be a finite subset of X,
  ${\cl A}/{I_F}$ is a $|F|$-idempotent operator algebra.
  Thus, it follows from the above lemma that ${\cl A}/{I_F}$
  is weak*-RFD, i.e., $\ \forall \ ([f_{ij}]) \in M_k({\cl A}/{I_F})$
  we have $\|([f_{ij}])\|= \sup_{n_F}\{\|(\rho_F([f_{ij}]))\|\}$
  where  supremum is taken over all weak*-continuous cc homomorphisms
  $ \rho_F:{\cl A}/{I_F} \rightarrow M_{n_F}$  and all integers $n_F.$\\
  \indent Let $(f_{ij}) \in M_k(\cl A),$ then $\|(f_{ij})\|_{M_k(\cl A)}= \sup_F \|([f_{ij}])\|$
  since $\cl A$ is local. Recall, that the weak*-topology on $\cl A$
  requires all the quotient maps of the form
  $\pi_F :\cl A \rightarrow {\cl A}/{I_F}, \ \pi_F(f)= [f]$ to be
  weak*-continuous. Thus for each finite subset $F \subseteq X$,  $\pi_F$
  is a weak*-continuous cc homomorphism.
  Let $\rho_F: {\cl A}/{I_F} \rightarrow M_{n_F} $ be a
  weak*-continuous cc homomorphism, then define
  $\delta_F= \pi_F \circ \rho_F$ which is indeed a
  weak*-continuous cc homomorphism. \\ Consider 
  \begin{multline*}
  \|(\rho_F([f_{ij}]))\|= \|(\rho_F(\pi_F(f_{ij})))\| = \|(\delta_F(f_{ij}))\| \leq \\
   \sup_n\{\|(\pi(f_{ij}))\|: \pi:\cl A \rightarrow M_n \text{ weak*-cont. cc homo}\}
  \leq  \|(f_{ij})\|_{M_k(\cl A)}. 
  \end{multline*}
  Finally, note that $\|([f_{ij}])\|_{M_k({\cl A}/{I_F})}=\sup_{n_F}\{\|(\rho_F([f_{ij}]))\|\},$ where supremum is taken over all weak*-continuous cc homomorphisms, $ \rho_F$ and positive integers $n_F.$  Hence, we obtain the
  result by taking supremum over all the finite subsets $F \subseteq X.$ 
  \end{proof}

  \begin{cor}
  Every local operator algebra of functions is RFD.
  \end{cor}
  \begin{proof}
   The proof of this follows immediately since every local operator 
   algebra is contained in some BPW complete local operator algebra
   completely isometrically.
  \end{proof}

\section{Quantized Function Theory on Domains}

  Whenever one replaces scalar variables by operator variables in a
  problem or definition, then this process is often referred to as {\em
  quantization.} It is in this sense that we would like to {\em
  quantize} the function theory on a family of complex domains. In
  some sense this process has already been carried out for balls in the
  work of Drury~\cite{Dr}, Popescu~\cite{Po}, Arveson~\cite{Ar}, and
  Davidson and  Pitts~\cite{DP}  and for polydisks in the work of
  Agler~\cite{Ag1},\cite{Ag2}, and Ball and Trent~\cite{BT}.
  Furthermore, the idea of ``quantizing'' other domains defined by
  inequalities occurs in \cite{AT}, \cite{BB}, and \cite{K-V}.
  We approach these same ideas via operator algebra methods.
  We will show that in many cases this process yields local
  operator algebras of functions to which the results of the
  earlier sections can be applied.

  We begin by defining a family of open sets for which our techniques
  will apply.

  \begin{defn}\label{ratpresdefn}
  Let $G \subseteq \bb C^N$ be an open set.
  If there exists a set of matrix-valued functions,
  $F_k=(f_{k,i,j}):G^- \to M_{m_k,n_k}, \ k \in I,$ whose components
  are analytic functions on $G,$ and $G = \{ z \in \bb C^N: \|F_k(z) \| < 1, k \in I \},$  
  then we call $G$ an {\bf analytically presented domain} and we call the
  set of functions $\cl R= \{F_k:G^- \to M_{m_k,n_k}: k \in I \}$ an
  {\bf analytic
  presentation of G.}
  The subalgebra $\cl A$ of the algebra of functions on $G$ generated by the
  component  functions $\{ f_{k,i,j}:  1 \le i \le m_k, 1 \le j \le
  n_k, k \in I \}$ and the
  constant function is called the {\bf algebra of the presentation.} 
 We say that $\cl R$ is a {\bf separating analytic presentation}
 provided that the algebra $\cl A$ separates points on $G.$
\end{defn}

 \begin{remark} An analytic presentation of $G$ by a finite set of matrix-valued functions, $F_k:G \to M_{m_k,n_k}, 1 \le k \le K,$  can always be replaced by the single block diagonal matrix-valued function, $F(z) = F_1(z) \oplus \cdots \oplus F_K(z)$
  into $M_{m,n}$ with $m = m_1+ \cdots m_K, n = n_1 + \cdots n_K$ and we will sometimes do this to simplify proofs. But it is often convenient to think in terms of the set, especially since this will explain the sums that occur in the Agler's factorization formula. 
\end{remark}

  Note that when we have a analytically presented domain, then every
  function in the algebra of the presentation is an analytic
  function on $G.$

  \begin{defn}\label{admissibledefn} 
  Let $G \subseteq \bb C^N$ be an analytically presented
  domain with presentation $F=(f_{i,j}): G \to M_{m,n},$
  and let $\cl H$ be a Hilbert space. A homomorphism $\pi: \cl A \to
  B(\cl H)$ of the algebra of the presentation is called an {\bf
    admissible representation} provided that
  $\|(\pi(f_{i,j}))\| \le 1$ in $M_{m,n}(B(\cl H) = B(\cl H^{n},
  \cl H^{m}).$ 
  We call the homomorphism $\pi$ an {\bf admissible strict
  representation} when these inequalities are all strictly less than
  1.
  Given $(g_{i,j}) \in M_n(\cl A)$ we set $\|(g_{i,j})\|_u = \sup \{
  \|(\pi(g_{i,j}))\| \},$ where the supremum is taken over all
  admissible representations $\pi$ of $\cl A.$ 
  We let $\|(g_{i,j})\|_{u_0}$
  denote the supremum that is obtained when we restrict to admissible strict representations.
  \end{defn}

  The theory of \cite{AT} and \cite{BB} studies domains defined as above with the additional restrictions that the set of defining functions is a finite set of polynomials, but they do not need their polynomials to separate points.  Our results should be compared to theirs.

  \begin{prop} Let G have a separating analytical presentation and let $\cl A$ be
  the algebra of the presentation. Then $\cl A$ endowed with either of the family of norms $\| \cdot \|_u$
  or $\| \cdot \|_{u_0}$ is an operator algebra of functions on $G$.
  \end{prop}
  \begin{proof} It is clear that it is an operator algebra and by
  definition it is an algebra of functions on G. It follows from the
  hypotheses that it separates points  of G. Finally, for every
  $\lambda = ( \lambda_1,..., \lambda_N) \in G,$ we have a
  representation of $\cl A$ on the one-dimensional Hilbert space given
  by $\pi_{\lambda}(f) = f(\lambda).$  Hence, $|f(\lambda)| \le
  \|f\|_u$ and so $\cl A$ is an operator algebra of functions on G.
  \end{proof}

  It will be convenient to say that matrices, $A_1, \ldots A_m$ are of
  {\em compatible sizes} if the product, $A_1 \cdots A_m$ exists, that
  is, provided that each $A_i$ is an $n_i \times n_{i+1}$ matrix.

  Given an analytically presented domain G, we include one extra function,
  $F_{k_0}$ which denotes the constant function 1.  By an {\bf admissible block
  diagonal matrix over G} we mean a block diagonal matrix-valued
  function of the form
  $D(z) = diag(F_{k_1},...,F_{k_m})$  where $k_i \in I\cup{k_0}$ for $0 \le i \le m.$
  Thus, we are allowing blocks of 1's in $D(z).$
  Finally, given a matrix $B$ we let $B^{(q)}= diag(B,....,B)$ denote
  the block diagonal matrix that repeats $B$ $q$ times.

  \begin{thm}\label{acontr} Let G be an analytically presented domain with
  presentation $\cl R = \{ F_k=(f_{k,i,j}): G \to M_{m_k,n_k}, k \in I
  \},$ let $\cl
  A$ be the algebra of the presentation and
  let $ P=(p_{ij}) \in M_{m,n}(\cl A),$ where $m,n$ are arbitrary. Then the following are equivalent:
  \begin{itemize}
  \item[(i)]$\|P\|_{u} < 1,$
  \item [(ii)] there exists an integer l, matrices of scalars $C_j, \  1
  \le j \le l$ with $\|C_j\| <1 \ $ and admissible block
  diagonal matrices $D_j(z), 1\leq j\leq l,$ which are of compatible sizes and
  are such that \[P(z)= C_1D_1(z) \cdots C_lD_l(z).\]
  \item [(iii)] there exists a positive, invertible matrix $ R \in
  M_{m}$ and matrices $P_0,P_k \in M_{m,r_k}(\cl A),  k \in K, $ where
  $K \subseteq I$ is a finite set, such that
  \[ I_m-P(z)P(w)^*= R + P_0(z)P_0(w)^* + \sum_{k \in K} P_k(z)(I- F_k(z)F_k(w)^*)^{(q_k)}P_k(w)^*\] 
  where $r_k = q_km_k$ and $z=(z_1,..., z_N), \ w=(w_1,..., w_N) \in G.$
  \end{itemize}
  \end{thm}

  \begin{proof} Although we will not logically need it, we first show
  that (ii) implies (i), since this is the easiest implication and
  helps to illustrate some ideas. Note that if $\pi:\cl A \to B(\cl H)$ is any
  admissible representation, then the norm of $\pi$ of any admissible block
  diagonal matrix is at most 1. Thus, if $P$ has the form of (ii),
  then for any admissible $\pi,$ we will have $(\pi(p_{i,j}))$
  expressed as a product of scalar matrices and operator matrices all
   of norm at most one and hence, $\|(\pi(P_{i,j}))\| \le
  \|C_1\| \cdots \|C_l\| <1.$  Thus, $\|P\|_u \le \|C_1\| \cdots
  \|C_l\| < 1.$

  We now prove that (i) implies (ii). The ideas of the proof are similar
  to \cite[Corollary~18.2]{Pa}, \cite[Corollary~2.11]{BP} and
  \cite[Theorem~1]{LMP} and use in an essential way the abstract
  characterization of operator algebras.
  For each $m,n \in \bb N, $ one proves that
  $\|P\|_{m,n} := inf\{ \|C_1\| \dots \|C_l\| \},$ defines a norm on $M_{m,n}(\cl A),$
  where the infimum is taken over all $l$ and all ways to factor $P(z)=C_1D_1(z_{i_1})\cdots C_lD_l(z_{i_l})$
  as a product of matrices of compatible sizes with scalar matrices
  $C_j, 1 \leq j \leq l$ and admissible block diagonal matrices $D_j, 1 \le j \le l.$

  Moreover, one can verify that $\cl M_{m,n}(\cl A)$ with this family $\{ \|.\|_{m, n}\}_{m,n}$ 
  of norms satisfies the axioms for an abstract unital operator algebra
  as given in \cite{BRS} and hence by the Blecher-Ruan-Sinclair
  representation theorem \cite{BRS}(see also \cite{Pa}) there exists a Hilbert space
  $\cl H$ and a unital completely isometric isomorphism $\pi : \cl A \longrightarrow \cl B(\cl H).\\$
  Thus, for every $m,n \in \bb N $ and for every $ P =(p_{ij})\in \cl M_{m,n}(\cl A),$
  we have $ \|P\|_{m,n} = \|(\pi(p_{ij}))\|.$ However, $\|\pi^{(m_k,n_k)}(F_k)\|=
  \|(\pi(f_{k,i,j}))\| \le 1$ for $1 \le i \le K,$ and
  so, $\pi$ is an admissible representation.
  Thus,  $\|P\|_{m,n} =\|(\pi(p_{ij}))\| \leq \|P\|_{u}.$ Hence, if
  $\|P\|_u <1,$ then $\|P\|_{m,n} < 1$ which implies that such a
  factorization exists.
  This completes the proof that (i) implies (ii).

  We will now prove that (ii) implies (iii).
  Suppose that $P$ has a factorization as in (ii). Let $K \subseteq I$
  be the finite subset of all indices that appear in the
  block-diagonal matrices appearing in the factorization of $P.$ We will use induction on $l$ to prove that (iii) holds.

  First, assume that $l=1$ so that
  $P(z) = C_1D_1(z).$
  Then,
  \begin{eqnarray*}
  I_m - P(z)P(w)^* &=& I_m - ( C_1D_1(z) ){( C_1D_1(w))}^* \nonumber \\
                   &=& \left(I_m - C_1C_1^*\right) + C_1\left(I - D_1(z){D_1(w)}^*\right)C_1^* .\label{a} 
  \end{eqnarray*}
  Since $D_1(z)$ is an admissible block diagonal matrix the $(i,i)$-th
  block diagonal entry of $I - D_1(z)D_1(w)^*$ is $I - F_{k_i}(z)F_{k_i}(w)^*$
  for some finite collection, $k_i.$
 
  Let $E_{k}$ be the diagonal matrix that has 1's wherever $F_k$
  appears(so $E_k=0$ when there is no $F_k$ term in $D_1$).
  Hence, $$C_1 (I - D_1(z)D_1(w)^*)C_1^*= \sum_k C_1E_k(I-F_{k}(z) F_k(w)^*) E_kC_1^*.$$
 
  Therefore, gathering terms for common values of $i,$
  \begin{eqnarray*}
  I_m - P(z)P(w)^* = R_0 + \sum_{k \in K}P_k(I - F_k(z) F_k(w)^*)P_k^*,
  \end{eqnarray*} where $R_0=I_m- C_1C_1^*$ is a positive, invertible
   matrix and  $ P_i$ is, in this case a constant. 
   Thus, the form (iii) holds, when $l=1$.

  We now assume that the form (iii) holds for any $R(z)$ that has a
  factorization of length at most $l-1,$ and assume that
  \[P(z) =C_1D_1(z) \cdots D_{l-1}(z) C_l D_{l}(z) = C_1 D_1(z)R(z),\]
  where $R(z)$ has a factorization of length $l-1.$  
 
  Note that a sum of expressions such as on the right hand side of (iii)
  is again such an expression. This follows by using the fact that
  given any two expressions $A(z), \ B(z),$ we can write 
  \begin{equation*}
  A(z)A(w)^* +B(z)B(w)^*=C(z)C(w)^*
  \end{equation*}
  where $C(z)= (A(z), \ B(z))$. 

  Thus, it will be sufficient to show that $I_m - P(z)P(w)^*$ is a sum of expressions as above.
  To this end we have that,
  \begin{multline*}
  I_m - P(z)P(w)^* = (I_m - C_1D_1(z)D_1(w)^*C_1^*) \\+ (C_1D_1(z))(I -
  R(z)R(w)^*)(D_1(w)^*C_1^*).
  \end{multline*}
  The first term of the above equation is of the form as on the
  right hand side of (iii) by case $l=1$. Also, the quantity
  $(I- R(z)R(w)^*)= R_0R_0^* + R_0(z)R_0(w)^* +\sum_{k \in K} R_k(z)(I- F_k(z)F_k(w)^*)^{(q_k)}R_k(w)^*$
  by the inductive hypothesis. 
  Hence, \begin{multline*} 
  C_1D_1(z)(I - R(z)R(w)^*)D_1(w)^*C_1^* =\\
  (C_1D_1(z)R_0)(C_1D(w)R_0)^* + [C_1D_1(z)R_0(z)][C_1D_1(w)R_0(w)]^*
  +\\ \sum_{k \in K}[C_1D_1(z)R_k(z)](I- F_k(z)F_k(w)^*)^{(q_k)}[C_1D_1(w)R_k(w)]^*.
  \end{multline*}

  Thus, we have expressed $(I- P(z)P(w)^*)$ as a sum of two terms both
  of which can be written in the form desired. Using again our remark
  that the sum of two such expressions is again such an expression, we
  have the required form. 

  Finally, we will prove (iii) implies (i).
  Let $\pi: \cl A \to B(\cl H)$ be an admissible
  representation and let $P=(p_{i,j}) \in M_{m,n}(\cl A)$ have a
  factorization as in (iii). To avoid far too many superscripts we write
  simplify $\pi^{(m,n)}$ to $\Pi.$  

  Now observe that
  \begin{multline*}
  I_m - \Pi(P)\Pi(P)^* = \Pi(R) + \Pi(P_0)\Pi(P_0)^*\\ + \sum_{k \in K} \Pi(P_k)(I-\Pi(F_k)\Pi(F_k)^*)^{(q_k)}(\Pi(P_k))^*.
  \end{multline*}
  Clearly the first two terms of the sum are positive. But since $\pi$ is
  an admissible representation, $\|\Pi(F_k)\| \le 1$ and hence, $(I-
  \Pi(F_k)\Pi(F_k)^*) \ge 0.$  Hence, each term on the right hand side
  of the above inequality is positive and since $R$ is strictly
  positive, say $R \ge \delta I_m$ for some scalar $\delta > 0,$ we have
  that $I_m - \Pi(P)\Pi(P)^* \ge \delta I_m.\\$
  Therefore, $\|\Pi(P)\| \le \sqrt{1- \delta}.$ 
  Thus, since $\pi$ was an arbitrary admissible representation, $\|P\|_{u} \le \sqrt{1- \delta} < 1,$ which proves (i).
  \end{proof}

  When we require the functions in the presentation to be row
  vector-valued, then the above theory simplifies somewhat and begins to
  look more familiar.
  Let G be an analytically presented domain with
  presentation $F_k: G \to M_{1,n_k}, k \in I.$ We identify
  $M_{1,n}$ with the Hilbert space $\bb C^n$ so 
  that $1 - F_k(z)F_k(w)^* = 1 - \langle F_k(z), F_k(w) \rangle,$ where
  the inner product is in $\bb C^{n}.$  In this case we shall say that 
  G is {\em presented by vector-valued functions.}

  \begin{cor}
  Let G be presented by vector-valued functions,
  $F_k=(f_{k,j}): G \to M_{1,n_k}, k \in I,$ let $\cl
  A$ be the algebra of the presentation and
  let $ P=(p_{ij}) \in M_{m,n}(\cl A).$ Then the following are equivalent:
  \begin{itemize}
  \item[(i)] $\|P\|_{u} < 1,$
  \item [(ii)] there exists an integer l, matrices of scalars $C_j, \  1
  \le j \le l$ with $\|C_j\| <1 \ $ and admissible block
  diagonal matrices $D_j(z), 1\leq j\leq l,$ which are of compatible sizes and
  are such that \[P(z)= C_1D_1(z) \cdots C_lD_l(z).\]
  \item [(iii)] there exists a positive, invertible matrix $ R \in
  M_{m}$ and matrices $P_0 \in M_{m, r_0}(\cl A),$ $P_k \in M_{m,r_k}(\cl A),  k\in K, $ where
  $K \subseteq I$ is finite, such that 
  \[I_m-P(z)P(w)^*= R + P_0(z)P_0(w)^* + \sum_{k \in K}(1- \langle
  F_k(z),F_k(w) \rangle) P_k(z)P_k(w)^*\] 
  where $z=(z_1,..., z_N), \ w=(w_1,..., w_N) \in G.$
  \end{itemize}
  \end{cor}

   The following result gives us a Nevanlinna-type result for the algebra of presentation. 
 
   \begin{thm}\label{Nevan} Let $Y$ be a finite subset of an
     analytically presented domain G with separating analytic presentation $F_k=(f_{k,i,j}): G
   \to M_{m_k,n_k}, k \in I,$  let $\cl A$ be the algebra of the presentation and let $P$ be a 
   $M_{m,n}$-valued function defined on a finite subset $Y=\{x_1,\cdots,x_l\}$ of $G.$
   Then the following are equivalent:
   \begin{itemize}
   \item[(i)] there exists $\tilde{P} \in M_{mn}(\cl A)$ such that 
   $\tilde{P}|_Y= P$ and $\|\tilde{P}\|_u < 1.$
   \item [(ii)] there exists a positive, invertible matrix $ R \in
   M_{m}$ and matrices $P_0 \in M_{m,r_0}(\cl A),$ $P_k \in M_{m,r_k}(\cl A),  k \in K, $
   where $K \subseteq I$ is a finite set, such that 
   \[ I_m-P(z)P(w)^*= R + P_0(z)P_0(w)^* + \sum_{k \in K} P_k(z)(I- F_k(z)F_k(w)^*)^{(q_k)}P_k(w)^*\]
 \noindent where $r_k = q_km_k$ and $z=(z_1,..., z_N), \ w=(w_1,..., w_N) \in Y.$
  \end{itemize}
  \end{thm}
  \begin{proof} 
  Note that $(i) \Rightarrow (ii)$ follows immediately as a corollary of Theorem~\ref{acontr}.
  Thus it only remains to show that $(ii) \Rightarrow (i).$ Since $\cl A$ is an operator algebra of 
  functions, therefore, ${\cl A}/{I_Y}$ is a finite dimensional operator algebra of idempotents 
  and ${\cl A}/{I_Y}= span\{E_1,\cdots, E_l\}$ for some $l= |Y|.$ Thus there exists a 
  Hilbert space, $H_Y$ and a completely isometric representation $\pi$ of ${\cl A}/{I_Y}$.
  By Theorem~\ref{tilde-multiplier}, there exists a kernel $K_Y$ such that 
  $\pi({\cl A}/{I_Y})= \cl M(K_Y)$ completely isometrically under the map 
  $\rho: \pi({\cl A}/{I_Y}) \to \cl M(K_Y)$ which sends $\pi(B)$ to $M_f,$ 
  where $B= \sum_{i=1}^l a_i\pi(E_i)$ and $f:Y \to \bb C$ is a function defined by $f(x_i)=a_i.$
  Note that $((I- F_k(x_i)F_k(x_j)) \otimes K_Y(x_i,x_j)) = ((I-\pi(F_k+I_Y)(x_i)\pi(F_k+I_Y)(x_j)^*) \otimes
  K_Y(x_i,x_j))_{ij} \ge 0 \ \forall \ k \in I$ 
  since $\| \pi(F_k+I_Y) \| \le \|F_k\|_u \le 1\ \forall \ k \in I.$ From this it follows that
   $((I_m-(\pi(P+I_Y))(x_i)(\pi(P+I_Y))(x_j)^*) \otimes K_Y(x_i,x_j))_{ij} \ge
   0.$ Using that $R > 0,$ we get that $\|\pi(P+I_Y)\| < 1.$ This shows that 
  there exists $\tilde{P} \in \cl A$ such that $\tilde{P}|_Y=P$ and $\|\tilde{P}\|_u < 1.$
  This completes the proof.
  \end{proof}

  We now turn towards defining quantized versions of the bounded
  analytic functions on these domains. For this we need to recall that
  the joint Taylor spectrum of a commuting $N$-tuple of operators
  $T=(T_1,...,T_N),$ is a compact set, $\sigma(T) \subseteq \bb C^N$
  and that there is an analytic functional calculus defined for any function
  that is holomorphic in a neighborhood of $\sigma(T).$

  \begin{defn} Let $G \subseteq \bb C^N$ be an analytically presented domain, with
  presentation $\cl R= \{F_k: G^- \to M_{m_k,n_k}, k \in I \}.$ We define
  the {\bf quantized version of G} to be the collection of all
  commuting N-tuples of operators,
  \begin{multline*}
\cl Q(G) =\\ \{T=(T_1,T_2,\ldots, T_N) \in B(\cl H): \sigma(T)
  \subseteq G \text{ and } \|F_k(T)\| \le 1, \forall k \in I \},
\end{multline*} where $\cl H$ is an arbitrary Hilbert space.
  We set \begin{multline*}
  \cl Q_{0,0}(G) =\\ \{ T=(T_1,T_2,\ldots,T_N) \in M_n:  \sigma(T)
  \subseteq G \text{ and } \|F_k(T)\| \le 1, \forall k \in I\},
  \end{multline*}
  where $n$ is an arbitrary positive integer.
  \end{defn}

  Note that if we identify a point $(\lambda_1,...,\lambda_N) \in \bb
  C^N$ with an N-tuple of commuting operators on a one-dimensional
  Hilbert space, then we have that $G \subseteq \cl Q(G).$

  If $T=(T_1,...,T_N) \in \cl Q(G),$ is a commuting N-tuple of operators
  on the Hilbert space $\cl H,$ then since the joint Taylor
  spectrum of $T$ is contained in G, we have that if $f$ is analytic on G, then there is an
  operator $f(T)$ defined and the map $\pi: Hol(G) \to B(\cl H)$ is a
  homomorphism, where $Hol(G)$ denotes the algebra of analytic functions
  on G.

  \begin{defn} Let $G \subseteq \bb C^N$ be an analytically presented domain, with
  presentation $\cl R= \{F_k: G^- \to M_{m_k,n_k}, \ k \in I \}.$  We define $H^{\infty}_{\cl R}(G)$ to be the set of
  functions $f \in Hol(G),$ such that $\|f\|_{\cl R} \equiv \sup \{
  \|f(T)\|: T \in \cl Q(G) \}$ is finite. Given $(f_{i,j}) \in
  M_n(H^{\infty}_{\cl R}(G)),$ we set $\|(f_{i,j})\|_{\cl R} = \sup \{
  \|(f_{i,j}(T))\|: T \in \cl Q(G) \}.$
  \end{defn}

  We are interested in determining when $H_{\cl R}^{\infty}(G) = \tilde{\cl
  A},$ completely isometrically and whether or not the $\|\cdot \|_{\cl R}$
  norm is attained on the smaller set $\cl Q_{0,0}(G).$

  Note that since each point in $G \subseteq \cl Q(G),$ we have that
  $H^{\infty}_ {\cl R}(G) \subseteq H^{\infty}(G),$ and $\|f\|_{\infty} \le
  \|f\|_{\cl R}.$  Also, we have that $\cl A \subseteq H^{\infty}_{\cl R}(G)$ and
  for $(f_{i,j}) \in M_n(\cl A),\   \|(f_{i,j})\|_{\cl R} \le \|(f_{i,j})\|_{u_0}\le \|(f_{i,j})\|_u .$ 
  Thus, the inclusion of $\cl A$ into $H^{\infty}_{\cl R}(G)$ might not even
  be isometric.

  \begin{thm}\label{mainthm} Let G be an analytically presented domain with a separating
  presentation $\cl R= \{ F_k:G \to M_{m_k,n_k}: k \in I
  \},$ let $\cl A$ be the algebra of the presentation and let
  $\widetilde{\cl A}$ be the BPW-completion of $\cl A.$  Then 
  \begin{itemize}
  \item[(i)] $\widetilde{\cl A} = H^{\infty}_{\cl R}(G),$ completely
  isometrically, 
  \item[(ii)] $H^{\infty}_{\cl R}(G)$ is a local, weak*-RFD, dual operator
  algebra,
  \end{itemize}
  \end{thm}
  \begin{proof} Let $f \in M_n(\widetilde{\cl A}),$ with $\|f\|_L <1.$
  Then there exists a net of functions, $f_{\lambda} \in M_n(\cl A),$
  such that $\|f_{\lambda}\|_u < 1$ and $\lim_{\lambda} f_{\lambda}(z)
  = f(z)$ for every $z \in G.$ Since $\|f_{\lambda}\|_{\infty} < 1,$
  by Montel's Theorem, there is a subsequence $\{ f_n \}$ of this net that converges to $f$ uniformly on compact sets. 
  Hence, if $T \in \cl Q(G),$ then $\lim_n \|f(T) - f_n(T)\| =0$ and so $\|f(T)\| \le \sup_n \|f_n(T)\| \le 1.$  Thus, we have that $f \in M_n(H^{\infty}_{\cl R}(G)),$ with $\|f\|_{\cl R} \le 1.$ This proves that $\widetilde{\cl A} \subseteq H^{\infty}_{\cl R}(G)$ and that $\|f\|_L \le \|f\|_{\cl R}.$

Conversely,  let $g \in M_n(H^{\infty}_{\cl R}(G)$ with $\|g\|_{\cl R} < 1.$  For any finite set $Y= \{ y_1,..., y_t \} \subseteq G,$ let $\cl A/ I_Y=  span \{ E_1,...,E_t \} $ be the corresponding $t$-idempotent algebra and let $\pi_Y: \cl A \to \cl A/I_Y$ denote the quotient map. Write $y_i = (y_{i,1},..., y_{i,N}), 1 \le i \le t$ and let $T_j = y_{1,j}E_1 + \cdots + y_{t,j} E_t, 1\le j \le N$ so that $T = (T_1,..., T_N)$ is a commuting $N$-tuple of opertors with $\sigma(T) = Y.$  For $k \in I,$ we have that $\|F_k(T)\| = \|F_k(y_1) \otimes E_1 + \cdots + F_k(y_t) \otimes E_t\| = \|\pi_Y(F_k) \| \le \|F_k\|_u =1.$  Thus, $T \in \cl Q(G),$ and so, $\|g(T) \| = \|g(y_1) \otimes E_1 + \cdots + g(y_t) \otimes E_t \| \le \|g\|_{\cl R}< 1.$  Since $\cl A$ separates points, we may pick $f \in M_n(\cl A)$ such that $f=g$ on $Y.$ Hence, $\pi_Y(f) = f(T) =g(T)$ and $\|\pi_Y(f)\| < 1.$ Thus, we may pick $f_Y \in M_n(\cl A),$ such that $\pi_Y(f_Y) = \pi_Y(f)$ and $\|f_Y\|_u < 1.$  This net of functions, $\{ f_Y \}$ converges to $g$ pointwise and is bounded.  Therefore, $g \in M_n(\widetilde{\cl A})$ and $\|g\|_L \le 1.$ This proves that $H^{\infty}_{\cl R}(G) \subseteq \widetilde{\cl A}$ and that $\|g\|_L \le \|g\|_{\cl R}.$

 Thus,  $\widetilde{\cl A} = H^{\infty}_{\cl R}(G)$ and the two matrix
 norms are equal for matrices of all sizes. The rest of the conclusions
 follow from the results on BPW-completions.
 \end{proof}

 \begin{remark}\label{finite} The above result yields that for every $f \in H^{\infty}_{\cl R}(G), 
  \ \|f\|_{\cl R}= \sup\{\|\pi(f)\| \}$ where supremum is taken over all finite dimensional 
  weak*-continuous representations, $\pi: H^{\infty}_{\cl R}(G) \to M_n \ \forall \ n.$
  Under additional hypotheses, we can show that $\|f\|_{\cl R}=\sup\{ \|f(T)\|: T \in \cl Q_{00}(G) \} 
  \ \forall f \in H^{\infty}_{\cl R}(G).$ Also, we can verify these
  hypotheses are met for most of the algebras given 
  in the example section. In particular, for Examples
  ~\ref{polydiskex},~\ref{ballrowex},~\ref{ballcolex},
  ~\ref{ballrowcolex}, \ref{lensex},~\ref{matrixballex}, and ~\ref{annulusex}. It would be interesting to know if 
  this can be done in general.
 \end{remark}

  We now seek other characterisations of the functions in
  $H^{\infty}_{\cl R}(G).$ In particular, we wish to obtain analogues of
  Agler's factorization theorem and of the results in \cite{AT} and \cite{BB}.
  By Theorem~\ref{tilde-multiplier}, if we are given an analytically presented domain $G \subseteq \bb C^N,$ with
  presentation $\cl R= \{F_k: G \to M_{m_k,n_k}, k \in I \},$ then
  there exists a Hilbert space $\cl H$ and a positive definite
  function, $K :G \times G \to B(\cl H)$ such that
  $\widetilde{\cl A}= \cl M(K).$  We shall denote any kernel
  satisfying this property by $K_{\cl R}.$

\begin{defn}  Let $G \subseteq \bb C^N$ be an analytically presented domain, with
  presentation $\cl R= \{F_k: G^- \to M_{m_k,n_k}, k \in I \}.$ We
  shall call a function $H: G \times G \to M_m$ an {\bf $\cl
    R$-limit,} provided that $H$ is the pointwise limit of a net of functions $H_{\lambda}:G
  \times G \to M_m$ of the form given by Theorem~\ref{acontr}(iii).
\end{defn}

\begin{cor}  Let $G \subseteq \bb C^N$ be an analytically presented
  domain, with a separating
  presentation $\cl R= \{F_k: G^- \to M_{m_k,n_k}, k \in I \}.$ Then
  the following are equivalent:
\begin{itemize}
\item[(i)]  $f \in M_m(H^{\infty}_{\cl R}(G))$ and $\|f\|_{\cl R} \le 1,$

\item[(ii)] $(I_m - f(z)f(w)^*) \otimes K_{\cl R}(z,w)$ is positive
  definite,

\item[(iii)] $I_m - f(z)f(w)^*$ is an $\cl R$-limit.
\end{itemize}
\end{cor}

In the case when the presentation contains only finitely many
functions we can say considerably more about $\cl R$-limits.

  \begin{prop}{\label{Uniform-bound}}
  Let G be an analytically presented domain with a finite
  presentation $\cl R = \{F_k=(f_{k,i,j}): G \to M_{m_k,n_k}, 1 \le k
  \le K \}.$
  For each compact subset $S \subseteq G,$ there exists a constant
  $C,$ depending only on $S,$ such that given a factorization of the form,
  \[ I_m-P(z)P(w)^*= R + P_0(z)P_0(w)^* + \sum_{k=1}^{K} P_k(z)(I- F_k(z)F_k(w)^*)^{(q_k)}P_k(w)^*,\]
  then $\|P_k(z)\| \le C$ $ \forall k, \forall z \in S.$.
  \end{prop}
  \begin{proof}
   By the continuity of the functions, there is a constant $\delta >
   0,$ such that $\|F_k(z)\| \le 1- \delta,$ $\forall k, \forall z \in
   S.$ Thus, we have that $I- F_k(z)F_k(z)^* \ge \delta I,$ $\forall \
   k, \forall z \in S.$
   Also, we have that 
  \begin{multline*}I_m \geq I_m-P(z)P(z)^* 
  \geq P_k(z)(I-F_k(z)F_k(z)^*)^{(q_k)}P_k(z)^* \geq \delta P_k(z)P_k(z)^*
  \end{multline*}
  This shows that  $\|P_k(z)\| \leq 1/{\delta} \ \forall k, \forall \ z \in S.$
  \end{proof}

The proof of the following result is essentially contained in \cite[Lemma~3.3]{BB}.

  \begin{prop}\label{F-limit}
  Let G be a bounded domain in $\bb C^N$ and let
  $F=(f_{i,j}): G \to M_{m,n}$ be analytic with $\|F(z)\|<1$ for $z
  \in G.$ If $H: G \times G \to M_p$ is analytic in the first variables, coanalytic
  in the second variables and there exists a net of 
  matrix-valued functions $P_{\lambda} \in M_{p,r_{\lambda}}(Hol(G))$
  which are uniformly bounded on compact subsets of G, such that 
  $H(z,w)$ is the pointwise limit of $H_{\lambda}(z,w) = P_{\lambda}(z)(I_m -
  F(z)F(w)^*)^{(q_{\lambda})}P_{\lambda}(w)^*$ where $r_{\lambda}=
  q_{\lambda}m_k,$ then there exists a Hilbert space $\cl H$ and an
  analytic function, $R:G \to B(\cl H \otimes \bb C^M, \bb C^p)$ such
  that
  $H(z,w) = R(z)[(I_m - F(z)F(w)^*) \otimes I_{\cl H}]R(w)^*.$  
  \end{prop}
  \begin{proof} We identify $(I_m - F(z)F(w)^*)^{(q_{\lambda})} = (I_m -
  F(z)F(w)^*) \otimes I_{\bb C^{q_{\lambda}}},$ and the $p \times
  mq_{\lambda}$ matrix-valued function $P_{\lambda}$ as an analytic
  function, $P_{\lambda}:G \to B(\bb C^m \otimes \bb C^{q_{\lambda}},
  \bb C^p).$ Writing $\bb C^m \otimes \bb C^{q_{\lambda}}= \bb
  C^{q_{\lambda}} \oplus \cdots C^{q_{\lambda}}$(m times) allows us to
  write $P_{\lambda}(z) = [ P^{\lambda}_1(z), \ldots ,
  P^{\lambda}_m(z)]$ where each $P^{\lambda}_i(z)$ is $p \times
  q_{\lambda}.$
  Also, if we let $f_1(z),...,f_m(z)$ be the $(1,n)$ vectors that
  represent the rows of the matrix $F,$ then we have that $F(z)F(w)^* =
  \sum_{i,j=1}^m f_i(z)f_j(w)^* E_{i,j}.$

  Finally, we have that \[H_{\lambda}(z,w) = \sum_{i=1}^m
  P^{\lambda}_i(z)P^{\lambda}_i(w)^* - \sum_{i,j=1}^m f_i(z)f_j(w)^*
  P^{\lambda}_i(z)P^{\lambda}_j(w).\]

  Let $K_{\lambda}(z,w) = (P^{\lambda}_i(z)P^{\lambda}_j(w)^*),$ so that
  $K_{\lambda}:G \times G \to M_m(M_p) = B(\bb C^m \otimes \bb C^p),$ is
  a positive definite function that is analytic in $z$ and co-analytic
  in $w.$ By dropping to a subnet, if necessary,
  we may assume that $K_{\lambda}$ converges uniformly on compact
  subsets of G to $K=(K_{i,j}):G \times G \to M_m(M_p).$ Note that this
  implies that $P^{\lambda}_i(z)P^{\lambda}_j(w)^* \to K_{i,j}(z,w)$ for
  all $i,j$ and that $K$ is a positive definite function that is analytic
  in $z$ and coanalytic in $w.$ 

  The positive definite $K$ gives rise to a reproducing kernel Hilbert
  space $\cl H$ of analytic $\bb C^m \otimes \bb C^p$-valued functions on G.
  If we let $E(z): \cl H \to \bb C^m \otimes \bb C^p,$ be the evaluation
  functional, then $K(z,w) = E(z)E(w)^*$ and $E:G \to B(\cl H, \bb C^m
  \otimes \bb C^p)$ is analytic.  Identifying $\bb C^m \otimes \bb C^p =
  \bb C^p \oplus \cdots \oplus \bb C^p$(m times), yields analytic
  functions,
  $E_i:G \to B(\cl H, \bb C^p),i=1,...,m,$ such that $(K_{i,j}(z,w))=
  K(z,w) = E(z)E(w)^* = (E_i(z)E_j(w)^*).$

  Define an analytic map $R:G \to B(\cl H \otimes \bb C^m, \bb C^p)$ by
  identifying $\cl H \otimes \bb C^m = \cl H \oplus \cdots \oplus \cl
  H$(m times) and setting $R(z)(h_1 \oplus \cdots \oplus h_m) =
  E_1(z)h_1 + \cdots +E_m(z)h_m.$
  Thus, we have that \begin{multline*}
  R(z)[(I_m - F(z)F(w)^*) \otimes I_{\cl H}] R(w)^*= \\ \sum_{i=1}^m E_i(z)E_i(w)^*
  - \sum_{i,j=1}^m f_i(z)f_j(w)^*E_i(z)E_j(w)^* =\\
  \sum_{i=1}^m K_{i,i}(z,w) - \sum_{i,j=1}^m f_i(z)f_j(w)^*K_{i,j}(z,w)
  =\\ \lim_{\lambda} \sum_{i=1}^m P_i^{\lambda}(z)P_i^{\lambda}(w)^* -
  \sum_{i,j=1}^m f_i(z)f_j(w)^* P_i^{\lambda}(z)P_j^{\lambda}(w)^* =
  H(z,w),
  \end{multline*}
  and the proof is complete.
  \end{proof}

  \begin{remark}\label{converse F-limit} 
  Conversely, any function that can be written in the
  form $H(z,w) = R(z)[(I_m - F(z)F(w)^*) \otimes I_{\cl H}]R(w)^*$ can
  be expressed as a limit of a net as above by considering the
  directed set of all
  finite dimensional subspaces of $\cl H$ and for each finite
  dimensional subspace setting $H_{\cl F}(z,w) = R_{\cl F}(z)[(I_m -
  F(z)F(w)^*) \otimes I_{\cl F}]R_{\cl F}(w)^*,$ where $R_{\cl F}(z) =
  R(z)(P_{\cl F} \otimes I_m)$ with $P_{\cl F}$ the orthogonal
  projection onto $\cl F.$
  \end{remark}

  \begin{defn} We shall refer to a function $H: G \times G \to
  M_m(M_p)$ that can
  be expressed as $H(z,w) = R(z)[(I_m - F(z)F(w)^*) \otimes \cl
  H]R(w)^*$ for some Hilbert space $\cl H$ and some analytic function
  $R:G \to B(\cl H \otimes \bb C^m, \bb C^p),$ as an {\bf F-limit.}
  \end{defn}

  \begin{thm}\label{tilde-Agler} Let G be an analytically
    presented domain with a finite separating
  presentation $\cl R = \{F_k=(f_{k,i,j}): G \to M_{m_k,n_k}, 1 \le
  k \le K \},$
  let $ f=(f_{ij})$ 
  be a $M_{m,n}$-valued function defined on G.
  Then the following are equivalent:
  \begin{itemize}
  \item[(1)] $f \in M_{mn}(H^{\infty}_{\cl R}(G))$ and $\|f\|_{\cl R} \le 1,$
  \item [(2)]there exists an analytic operator-valued
  function $R_0:G \to B(\cl H_0, \bb C^m)$ and $F_k$-limits, $H_k:G
  \times G \to M_m,$ such that \[I- f(z)f(w)^* = R_0(z)R_0(w)^* +
  \sum_{k=1}^K H_k(z,w) \ \forall \ z,w \in G,\]
\item[(3)] there exist $F_k$-limits, $H_k(z,w),$ such that
\[I- f(z)f(w)^* = \sum_{k=1}^K H_k(z,w) \ \forall z,w \in G.\]
  \end{itemize}
  \end{thm}
  \begin{proof} Recall that $\tilde{\cl A} = H^{\infty}_{\cl R}(G).$
  Let us first assume that $f \in M_{mn}(\tilde{A})$ and $\|f\|_{M_{mn}(\tilde{A})} < 1.$
  Then for each finite set $Y$, there exists $f_Y \in M_{mm}(\cl A)$ such that $ f_Y$ converges to f pointwise and 
  $\|f_Y\|_u \leq 1.$ We may assume that $\|f_Y\|_u < 1$ 
  by replacing $f_Y$ by $\frac{f_Y}{1+1/|Y|},$
  where $|Y|$ denotes the cardinality of the set $Y$.\\
  \indent Thus by Theorem~\ref{acontr}
  there exists a positive, invertible matrix $ R^Y\in
  M_{m}$ and matrices $P^Y_k \in M_{m,r_{k_Y}}(\cl A),  0 \leq k \leq K, $ such that 
  \[ I_m-f_Y(z)f_Y(w)^*= R^Y + P^Y_0(z)P^Y_0(w)^* + \sum_{k=1}^{K} P^Y_k(z)(I- F_k(z)F_k(w)^*)^{(q_{k_Y})}P^Y_k(w)^*\]
  where $r_{k_Y} = q_{k_Y}m_k$ and $z, \ w \in G.$ 
  If we define a map $F_0:G \to M_{m_0,n_0}$ as the zero map 
  then the above factorization can be written as 
  \[I_m-f_Y(z)f_Y(w)^*= R^Y + \sum_{k=0}^{K} P^Y_k(z)(I- F_k(z)F_k(w)^*)^{(q_{k_Y})}P^Y_k(w)^*\] 
  where $r_{k_Y} = q_{k_Y}m_k$ and $z, \ w \in G.$ \\
  \indent Note that the net $R^Y$  is uniformly bounded above by 1, 
  thus there exists $R\in M_m$ and a subnet $R^{Y_s}$ which converges to R.\\
  
  Finally, since the net $f_Y$ converges to f pointwise we have that the net 
  $\sum_{k=1}^{K} P^Y_k(z)(I- F_k(z)F_k(w)^*)^{(q_{k_Y})}P^Y_k(w)^*$
  converges pointwise on $G$. Also note that for each k, $<P_k^Y>$ is the net of 
  vector-valued holomorphic function and is uniformly bounded on compact subsets of $G$
  by Proposition~\ref{Uniform-bound}. \\
  
  Thus by Proposition~\ref{F-limit} there exists $F_k$-limit
  for each $0 \le k \le K,$ that is, there exists K+1 Hilbert spaces $\cl H_k$ and K+1
  analytic function, $R_k:G \to B(\cl H_k \otimes \bb C^M, \bb C^p)$ such
  that \\ $H_k(z,w) = R_k(z)[(I_m - F_k(z)F_k(w)^*) \otimes I_{\cl H_k}]R_k(w)^*$
  and the corresponding subnet of the net  $\sum_{k=0}^{K} P^Y_k(z)(I- F_k(z)F_k(w)^*)^{(q_{k_Y})}P^Y_k(w)^*$
  converges to  $\sum_{k=0}^K H_k(z,w) \ \forall \ z,\ w \in G.$ This completes the proof that (1) implies (2). 

  \indent To show the converse, assume that there exists an analytic operator-valued
  function $R_0:G \to B(\cl H_0, \bb C^m)$ and K analytic functions, 
  $R_k:G \to B(\cl H_k \otimes \bb C^M, \bb C^p)$ on some Hilbert space $H_k$
  such that $I- f(z)f(w)^* = R_0(z)R_0(w)^* + \sum_{k=1}^K R_k(z)(I- F_k(z)F_k(w)^*)^{(q_{k})}R_k(w)^* \ \forall \ z,w \in G.$ 
  By using Theorem~\ref{tilde-multiplier} there exists a vector-valued Kernel $K$ such that
  $M_n(\cl M(K))=M_n(\tilde{A})$ completely isometrically for every n. It is easy to see that
  $((I- f(z)f(w)^*) \otimes K(z,w)) \geq 0 \ \forall z,\ w \in G.$ This is equivalent to 
  $f\in M_m(\cl M(K))$ and $\|M_f\| \leq 1$ which in turn is
  equivalent to (1). Thus, (1) and (2) are equivalent.

Clearly, (3) implies (2). The arguement for why (2) implies (3) is
contained in \cite{BB} and we recall it. If we fix any $k_0,$ then
since $\|F_{k_0}(z)\|<1$ on $G,$ we have that $|f_{k_0,1,1}(z)|^2 +
\cdots + |f_{k_o,1,m}(z)|^2 < 1$ on $G.$ From this it follows that
$H(z,w) = (1
- f_{k_0,1,1}(z)\bar{f_{k_0,1,1}(w)} - \cdots -
f_{k_0,1,m}(z)\bar{f_{k_0,1,m}(w)})$ is an $F_{k_0}$-limit and that
$H^{-1}(z,w)$ is positive definite. Now we have that
 $R_0(z)R_0(w)^*H^{-1}(z,w)$ is positive definite and so we may write, \\
 $R_0(z)R_0(w)^*H^{-1}(z,w)= G_0(z)G_0(w)^*$ and we have that
 $R_0(z)R_0(w)^* = G_0(z)H(z,w)G(w)^.$ This shows that
 $R_0(z)R_0(w)^*$ is an $F_k$-limit and so it may be absorbed into the sum.
  \end{proof}



\section{Examples and Applications}

In this section
we present a few examples to illustrate the above definitions and results.

  \begin{ex}\label{polydiskex} Let $G= \bb D^N,$ be the polydisk which has a
  presentation given by the coordinate functions $F_i(z) = z_i,
  1 \le i \le N.$ Then the algebra of this presentation is the algebra
  of polynomials and an admissible representation is given by any
  choice of N commuting contractions, $(T_1,...,T_N)$ on a Hilbert
  space. Given a matrix of polynomials, $\|(p_{i,j})\|_u= \sup
  \|(p_{i,j}(T_1,...,T_N))\|$ where the
  supremum is taken over all N-tuples of commuting contractions.  This
  is the norm considered by Agler in \cite{Ag1}, which is sometimes
  called the {\em Schur-Agler} norm \cite{LMP}. 
Our $\cl Q(\bb D^N) = \{ T= (T_1,...,T_N): \sigma(T) \subseteq
\bb D^N \text{ and } \|T_i\| \le 1 \}.$ Note that if we replace such a
$T$ by $rT= (rT_1,..., rT_N)$ then $\|rT_i\| < 1,$ $rT \in \cl Q_{\cl
  R}(\bb D^N)$ and taking suprema
over all $T \in \cl Q_{\cl R}(\bb D^N)$ will be the same as taking a
suprema over this smaller set. Thus, the algebra
  $H^{\infty}_{\cl R}(\bb D^N)$ consists of those analytic functions $f$
  such that
\[ \|f\|_{\cl R} = \sup \{ \|f(T_1,...,T_N)\|: \|T_i\|<1, i=1,...,N \}
< +\infty. \]
By Theorem~\ref{tilde-Agler} for $f \in M_{m,n}(H^{\infty}_{\cl R}(\bb
D^N),$ we have that $\|f\|_{\cl R} \le 1$ if and
only if
\[ I_m - f(z)f(w)^* = \sum_{i=1}^N (1-
z_i\bar{w_i})K_i(z,w),\]
for some analytic-coanalytic positive definite functions, $K_i:\bb D^N
\times \bb D^N \to M_m.$ 
  \end{ex}

  \begin{ex}\label{ballrowex} Let $G= \bb B_N$ denote the
  unit Euclidean ball in $\bb C^N.$ If we let $F_1(z)=(z_1,...,z_N):
   \bb B_N^- \to M_{1,N},$ then this gives us a polynomial presentation.
  Again the algebra of the presentation is the polynomial algebra.
  An admissible representation corresponds to an N-tuple of commuting
  operators $(T_1,...,T_n)$ such that $T_1T_1^* + \cdots + T_NT_N^* \le I,$
  which is commonly called a {\em row contraction} and an admissible strict
  representation is given when $T_1T_1^* + \cdots + T_NT_N^* < I,$
  which is generally referred to as a {\em strict row contraction.} In
  this case one can again easily see that $\| \cdot \|_u = \| \cdot
  \|_{u_0}$ by using the same $r<1$ argument as in the last example
  and that $f \in H^{\infty}_{\cl R}(\bb B_N)$ if and only if
\[ \|f\|_{\cl R} = \sup \{\|f(T)\|: T_1T_1^*+ \cdots +T_NT_N^* < I\}< +\infty \]
  These are the norms on polynomials considered by  Drury\cite{Dr},
  Popescu~\cite{Po}, Arveson~\cite{Ar}, and Davidson and Pitts~\cite{DP}. 
  By Theorem~\ref{tilde-Agler} we will have for $f \in
  M_{m,n}(H^{\infty}_{\cl R}(\bb B_N))$ that $\|f\|_{\cl R} \le 1$ if
  and only if
\[ I_m - f(z)f(w)^* = (1 - \langle z,w \rangle)K(z,w), \]
where $K: \bb B_N \times \bb B_N \to M_m$ is an analytic-coanalytic
positive definite function. \end{ex}

  \begin{ex}\label{ballcolex} Let $G= \bb B_N$ as above and let $F_1(z) =
  (z_1,...,z_N)^t: \bb B_N \to M_{N,1}.$ Again this is a rational
  presentation of G and the algebra of the
  presentation is the polynomials. An admissible representation
  corresponds to an N-tuple of commuting operators $(T_1,...,T_N)$
  such that $\|(T_1,...,T_N)^t\| \le 1,$ i.e., such that $T_1^*T_1+
  \cdots + T_N^*T_N \le I,$ which is generally referred to as a {\em
   column contraction.} This time the norm on $H^{\infty}_{\cl R}(\bb
 B_N)$ will be defined by taking suprema over all strict column
 contractions and we will have that $\|f\|_{\cl R} \le 1$ if and only
 if
\[ I_m - f(z)f(w)^* = R_1(z)[(I_N - (z_i\bar{w_j})) \otimes
I_{\cl H}]R_1(w)^* \]
for some $R_1:\bb B_N \to
B(\bb C^m, \cl H),$ analytic.
   \end{ex}

\begin{ex}\label{ballrowcolex} Let $G= \bb B_N$ as above, let $F_1(z) = (z_1,...,z_N):
  \bb B_N^- \to M_{1,N}$ and $F_2(z) =
  (z_1,...,z_N)^t: \bb B_N \to M_{N,1}.$ Again this is a rational
  presentation of G and the algebra of the
  presentation is the polynomials. An admissible representation
  corresponds to an N-tuple of commuting operators $(T_1,...,T_N)$
  such that $T_1T_1^*+ \cdots + T_NT_N^* \le I$ and $T_1^*T_1+
  \cdots + T_N^*T_N \le I,$ that is, which is both a row and column
  contraction. This time the norm on $H^{\infty}_{\cl R}(\bb B_N)$ is
  defined as the supremum over all commuting N-tuples that are both
  strict row and column contractions. We will have that $f \in
  M_{m,n}(H^{\infty}_{\cl R}(\bb B_N))$ with $\|f\|_{\cl R} \le 1$ if
  and only if
\[ I_m - f(z)f(w)^* = (1 - \langle z,w \rangle)K_1(z,w) +
R_1(z)[(I_N - (z_i \bar{w_j})) \otimes I_{\cl H} ]R_1(w)^*, \] where
$K_1$ and $R_1$ are as before.
\end{ex}

The last three examples illustrate that it is possible to have multiple
rational representations of $G,$ all with the same algebra, but which
give rise to (possibly) different operator algebra norms on $\cl A.$
Thus, the operator algebra norm depends not just on $G,$ but also on the
particular presentation of $G$ that one has chosen.  We
have surpressed this dependence on $\cl R$ to keep our notation
simplified.

  \begin{ex}\label{funnydisk}
  Let $G= \bb D$ the open unit disk in the complex plane and 
  let $F_1(z) = z^2, F_2(z) = z^3.$ It is easy to check that 
  the algebra $\cl A$ of this presentation is generated by the component functions and 
  the constant function is the span of the monomials, 
  $\{ 1, z^n: n \ge 2 \},$ and that $\cl A$ separates the points of $G$. In this case an (strict)admissible 
  representation, $\pi: \cl A \to B(\cl H),$ is given by any 
  choice of a pair of commuting (strict)contractions, $A= \pi(z^2), B= \pi(z^3),$ 
  satisfying $A^3=B^2.$ Again, it is easy to see that $\|.\|_u=\|.\|_{u_0}.$
  On the other hand \[ \cl Q(\bb D) = \{ T: \sigma(T) \subseteq
  \bb D \text{ and } \|T^2\| \le 1, \|T^3\| \le 1 \} \] and it can be seen that 
$H^{\infty}_{\cl R}(\bb D)$ is defined by
\[ \|f\|_{\cl R} = \sup \{ \|f(T)\|: \|T^2\|<1,\|T^3\| < 1 \} <
+\infty. \]

In this case we have that $f \in M_{m,n}(H^{\infty}_{\cl R}(\bb D)$
and $\|f\|_{\cl R} \le 1$ if and only if
\[ I_m - f(z)f(w)^* = (1-z^2\bar{w^2})K_1(z,w) + (1-z^3
\bar{w^3})K_2(z,w). \]
\end{ex}
 

  \begin{ex}\label{lensex} Let $\bb L = \{ z \in \bb C: |z-a| <1, |z-b| <1 \},$ where
  $|a-b| <1,$ then the functions $f_1(z) = z-a, f_2(z) = z-b$ give a
  polynomial presentation of this ``lens''. The algebra of this
  presentation is again the algebra of polynomials. An admissible
  representation of this algebra is defined by choosing any operator
  satisfying $\|T - aI \| \le 1$ and $\|T - bI \| \le 1,$ with strict
  inequalities for the admissible strict representations. In this
  case we easily see that $\| \cdot \|_u = \| \cdot \|_{u_0},$
  since given any operator T satisfying $\|T - aI \| \le 1$ and
  $\|T - bI \| \le 1,$ and $r<1,\  S_r= rT+(1-r)(a+b)$ corresponds to 
  the admissible strict representations and for any matrix of polynomials
  $\|(p_{i,j}(T))\| =  \lim_{r\rightarrow 1}\|(p_{ij}(S_r))\|.$ This
  algebra with this norm was studied in \cite{BC}. Their work shows that
  this norm is completely boundedly equivalent to the usual supremum norm.
  Our results imply that $f \in M_{m,n}(H^{\infty}_{\cl R}(\bb L))$
  and $\|f\|_{\cl R} \le 1$ if and only if
\[ I_m -f(z)f(w)^* = (1 - (z-a) (\overline{w-b}))K_1(z,w)
+ (1-(z-b)(\overline{w-b}))K_2(z,w). \]
\end{ex}

  \begin{ex}\label{matrixballex} Let $G= \{ (z_{i,j}) \in M_{M,N} : \|(z_{i,j}) \| < 1 \}$ and
  let $F:G \to M_{M,N}$ be the identity map $F(z) = (z_{i,j}).$ Then this
  is a polynomial presentation of G and the algebra of the presentation
  is the algebra of polynomials in the $MN$ variables $\{ z_{i,j}
  \}.$
  An admissible representation of this algebra is given by any choice of
  $MN$ commuting operators $\{ T_{i,j} \}$ on a Hilbert space $\cl H,$
  such that $\|(T_{i,j})\| \le 1$ in $M_{M,N}(B(\cl H))$ and as above,
  one can show that $\| \cdot \|_{\cl R}$ is acheived by taking
  suprema over all commuting $MN$-tuples for which $\|(T_{i,j})\|< 1.$
We have that $f \in M_{m,n}(H^{\infty}_{\cl R}(G)$ and $\|f\|_{\cl R}
\le 1$ if and only if
\[ I_m - f(z)f(w)^* = R_1(z)[(I_M - (z_{i,j})(w_{i,j})^*)
\otimes I_{\cl H}]R_1(w)^*, \] for some appropriatley chosen $R_1.$
  \end{ex}

All of the above examples are also covered by the theory of \cite{AT}
and \cite{BB}, except that their definition of the norm is slightly
different. We address this difference in a later remark. We now turn to some examples that are not covered by
their theory.

  \begin{ex}\label{annulusex} Let $0<r <1$ be fixed and
  let $\bb A_r = \{ z \in \bb C: r<|z| <1 \}$ be an annulus.
  Then this has a rational presentation given by $F_1(z) = z$ and
  $F_2(z) = r z^{-1},$ and the algebra of this presentation is just
  the Laurent polynomials. Admissible representations of this algebra
  are given by selecting any invertible operator $T$ satisfying
  $\|T\| \le 1$ and $\|T^{-1}\| \le r^{-1}.$
  Admissible strict representations are given by invertible
  operators satisfying $\|T\| < 1$ and $\|T^{-1}\| < r^{-1}.$ It is no
  longer quite so clear that $\| \cdot \|_u = \| \cdot \|_{u_0}.$
  However, this algebra with these norms is studied by the first author
  in \cite{Mi} and among other results the equality of these norms was
  shown. In \cite{DP}, it was shown that $\| \cdot \|_u$ is completely
  boundedly equivalent to the usual supremum norm. In this case one can
  see that the $\|\cdot\|_{\cl R}$ is achieved by taking suprema over
  all $T$ satisfying $\|T\| <1$ and $\|T^{-1}\| < r^{-1}.$ The formula
  for the norm is given by $\|f\|_{\cl R} \le 1$ if and only if
  \[ I_m -f(z)f(w)^* = (1 -z \bar{w})K_1(z,w) +(1-r^2 z^{-1}
  \overline{w^{-1}})K_2(z,w). \]
  \end{ex}

  \begin{ex}\label{halfplane} Let G be a simply connected domain in $\bb C$ 
   and $\phi: G \to \bb D$ be a biholomorphic map. Then $G=\{z \in \bb C: |\phi(z)|< 1\}$ and the 
   $\cl Q(G)= \{T : \sigma(T)\subseteq G \text{ and }
   \|\phi(T)\| \le 1\}$ where $\cl R= \{ \phi\}.$ In this case 
  the algebra $\cl A$ of the presentation is just the algebra of all
  polynomials in $\phi.$ Thus, an admissible representation of this algebra is defined by 
  choosing an operator $B \in B(\cl H)$ that satisfies $\|B\|\le 1$ and defining
  $\pi: \cl A \to B(\cl H)$ via $\pi( p(\phi)) = p(B),$ where $p$ is a
  polynomial.  A strict
  admissible representation is defined similarly by first choosing a strict contraction.
  In this case, it is immediate that $\|.\|_u=\|.\|_{u_0}$ and that $f \in H^{\infty}_{\cl R}(G)$ 
  if and only if $f \in Hol(G)$ and \[ \|f\|_{\cl R}= \sup \{ \|f(T)\|: T \in \cl Q_{\cl R}(G)\} < +\infty.\]
  Our results imply that $\|f\|_{\cl R} \le 1 $ if and only if
        \[1-f(z)f(w)^* = (1-\phi(z)\phi(w)^*)K_1(z,w)\]
  In particular, if we take $\phi(z)= \frac{z-1}{z+1}$ then it maps
  the halfplane $\bb H= \{z: Re(z) > 0\}$ to the unit disk. For this
  particular $\phi,$ we have that $\cl Q_{\cl R}= \{ T: \sigma(T)
  \subseteq \bb H \text{ and } Re(T) \ge 0 \}.$
\end{ex}

\begin{ex}
Similarly, if we let $G= \{ z \in \bb C^N: |\phi_i(z)| < 1, i=1,...,N
\}$ where $\phi_i(z) = \frac{z_i -1}{z_i+1},$ then $G$ is an
intersection of half planes and $\cl Q_{\cl R}$ consists of all
commuting $N$-tuples of operators, $(T_1,...,T_N)$ such that
$\sigma(T_i) \subseteq \bb H$ and $Re(T_i) \ge 0$ for all $i.$
 Applying our results, we obtain a factorization result for half planes. 
  These algebras have been studied by D.Kalyuzhnyi-Verbovetzkii in \cite{K-V}.
  \end{ex}

\begin{ex}\label{convexset}
Let $G \subseteq \bb C$ be an open convex set and represent it as an
intersection of half planes $\bb H_{\theta}.$ Each half plane can be expressed as $\{z:
|F_{\theta}(z)| < 1 \}$ for some family of linear fractional
maps. If we let $\cl R = \{ F_{\theta} \},$ then $\cl Q_{\cl R}(G) = \{T: \sigma(T) \subseteq G \text{ and }
\|F_{\theta}(T)\| \le 1 \ \forall \theta \}.$ Moreover, each inequality
$\|F_{\theta}(T)\| \le 1$ can be re-written as an operator inequality
for the real part of some translate and rotation of $T.$
For example, when $G = \bb D,$ we may take $F_{\theta}(z) =
\frac{z}{z- 2 e^{i \theta}},$ for $0 \le \theta < 2\pi.$ In this case,
one checks that $\|F_{\theta}(T) \| \le 1$ if and only if $Re(e^{i
  \theta}T) \le I.$ Thus, it follows that
\[ \cl Q(\bb D) = \{ T: \sigma(T) \subseteq \bb D \text{ and }
w(T) \le 1 \}, \]
where $w(T)$ denotes the numerical radius of $T.$
Thus, $H^{\infty}_{\cl R}(\bb D)$ becomes the ``universal'' operator
algebra that one obtains by substituting an operator of numerical
radius less than one for the variable $z$ and we have a quite
different quantization of the unit disk.
Our results give a formula for this norm, but only in terms
of $\cl R$-limits, so further work would need to be done to make it
explicit.
\end{ex} 

\begin{ex}\label{convexset2} There is a second way that
  one can quantize many convex sets. Let $G= \{ z: |z - a_k| <
  r_k, k \in I \} \subseteq \bb C$ be an open, bounded convex set that
  can be expressed as an intersection of a possibly infinite set of
  open disks. For example, any convex set with a smooth boundary with
  uniformly bounded curvature can be expressed in
  such a fashion. Then $G$ has a rational presentation given by
  $F_k(z) = r_k^{-1}(z - a_k),k \in I$ the algebra of the presentation is
  just the polynomial algebra and an admissible representation is
  given by selecting any operator $T$ satisfying, $\|T - a_kI \| \le
  r_k, k \in I.$ Thus, we again a factorization result, but only
  in terms of $\cl R$-limits.
\end{ex}

  The above definitions allow one to consider many other examples. For
  example, one could fix $0<r<1$ and let $G = \{ z \in \bb B_N: r<|z_1| \},$ with
  rational presentation $f_1(z) = (z_1,...,z_N) \in M_{1,N},$ and $f_2(z) =
  rz_1^{-1}.$ An admissible representation would then correspond to a commuting
  row contraction with $T_1$ invertible and $\|T_1^{-1}\| \le r^{-1}.$\\

We now compare and contrast some of our hypotheses with those of
\cite{AT} and \cite{BB}.
  
  \begin{remark}
  Let $G=\{z \in \bb C^N: \|F_k(z)\| < 1 \ \forall \ \ k=1,\cdots, K\}$ 
  where the $F_k$'s are matrix-valued polynomials defined on $G$. 
  Then for $ f \in Hol(G),$ \cite{AT} and \cite{BB} really study a
  norm given by 
  $\|f\|_{s}= sup\{\|f(T)\|\}$ where the supremum is taken over all
  commuting $N$-tuples of 
  operators $T$ with $ \|F_k(T)\| < 1\ \forall k.$ We wish to contrast
  this norm with our $\|f\|_{\cl R}.$ In \cite{AT} it is
  shown that the hypotheses $\|F_k(T)\| < 1, k=1,...,K$ implies that
  $\sigma(T) \subseteq G.$  Thus, we have that $\|f\|_s \le \|f\|_{\cl
    R}.$  In fact, we
  have that $\|f\|_s = \|f\|_{\cl R}.$  This can be seen by the fact
  that they obtain identical factorization theorems to
  ours. This can also be seen directly in some cases where the algebra $\cl A$
  contains the polynomials and when it can be seen that $\|\cdot
  \|_{\cl R}$ is attained by taking the supremum over matrices(see Remark~\ref{finite}).
  Indeed, if $\|f\|_{\cl R}$ is attained as the supremum over commuting
  $N$-tuples of finite matrices $T= (T_1,...,T_N)$ satisfying $\sigma(T) \subseteq G$ and
  $\|F_k(T)\| \le 1$ then such an $N$-tuple of commuting matrices, can be conjugated by a
  unitary to be simultaneously put in upper triangular form.  Now it is easily argued that the
  strictly upper triangular entries can be shrunk slightly so that one
  obtains new $N$-tuples $T_{\epsilon}=(T_{1,\epsilon},...,T_{N,\epsilon})$ satisfying,
  $\|F_k(T_{\epsilon})\| < 1, k=1,...,K$ and $\|T_i - T_{i,\epsilon}\|
  < \epsilon.$ But we do not have a simple direct argument that works
  in all cases.
  \end{remark}
  
\begin{remark} We do not know how generally it is the case that $\|
  \cdot \|_u$ is a local norm.  That is, we do not know if $\|f
  \|_u = \|f \|_{\cl R}$ for $f \in M_n(\cl A).$  In particular, we do
  not know if this is the case for Example~\ref{funnydisk}.
In this case, the algebra of the of the presentation is $\cl A = span
\{ z^n: n \ge 0, n \ne 1 \}.$ If we write a polynomial $p \in \cl A$
in terms of its even and odd decomposition, $p = p_e + p_o,$ then
$p_e(z) = q(z^2)$ and $p_o = z^3r(z^2)$ for some polynomials $q,r.$ In
this case it is easily seen that
\[ \|p\|_u = \sup \{ \|q(A) + Br(A)\|: \|A\|\le 1, \|B\| \le 1, AB=BA, A^3
= B^2 \}, \] while
\[ \|p\|_L=\|p\|_{\cl R} = \sup \{ \|p(T) \|: \|T^2\| \le 1, \|T^3\| \le 1 \}. \]

\end{remark}


\end{document}